%% file: template.tex
\documentclass{m2an}
\usepackage{cases}
\usepackage{enumerate}

\begin{document}
\title{A cut finite element method for the Darcy problem}\thanks{We gratefully acknowledge Professor Erik Burman for the many precious suggestions and comments.}\thanks{This research was partially supported by ERC AdG project CHANGE n. 694515.}
\author{Riccardo Puppi}\address{Chair of Modelling and Numerical Simulation, École Polytechnique Fédérale de Lausanne, Lausanne, CH (riccardo.puppi@epfl.ch)}
%
\date{\today}
\begin{abstract} 
We present and analyze a cut finite element method for the weak imposition of the Neumann boundary conditions of the Darcy problem. The Raviart-Thomas mixed element on both triangular and quadrilateral meshes is considered. Our method is based on the Nitsche formulation studied in~\cite{puppi_burman} and can be considered as a first attempt at extension in the unfitted case. The key feature is to add two ghost penalty operators to stabilize both the velocity and pressure fields. We rigorously prove our stabilized formulation to be well-posed and derive \emph{a priori} error estimates for the velocity and pressure fields. We show that an upper bound for the condition number of the stiffness matrix holds as well. Numerical examples corroborating the theory are included.
\end{abstract}
%
%
\maketitle
\section*{Introduction}
Mesh generation is one of the major bottlenecks for the classical finite element method  for the numerical solution of partial differential equations (PDEs). It is a very costly process because it necessitates not only of important computational power, but also of ad hoc human intervention. Mathematical problems for which mesh generation becomes a very demanding task are, in general, problems set in complicated geometries or such that the domain changes during the simulation. The last few decades have seen the flourishing of an abundance of numerical methods trying to overcome this issue. The most popular approaches belong to the family of the so-called fictitious domain methods (from the pioneering work~\cite{PESKIN1972252}) where the possibly complicate domain $\Omega$ is immersed in a much simpler geometry $\Omega_{\mathcal T}$, for which the generation of the mesh is a simple task. In the last few years, the Cut Finite Element Method (CutFEM)~\cite{MR3416285}, a particular fictitious domain method, gained a lot of attention and showed its potential in different applications in science and engineering~\cite{MR3239219,MR3489084,MR3402347}. Unlike other methods, the CutFEM relies on solid theoretical foundations, and its key feature is to add weakly consistent stabilization operators~\cite{MR2738930} to the variational formulation of the discrete problem to transfer the stability and approximation properties from the finite element scheme constructed on the background mesh to its cut finite element counterpart. 

In this contribution we study the weak imposition of the Neumann boundary conditions for the Darcy problem when the mesh does not fit the boundary of the domain. Let us recall that the Darcy problem combines a constitutive equation describing the flow of a fluid in porous media with a conservation of mass equation, and it is coupled with suitable boundary conditions. This first-order system of PDEs can also be derived as a dual formulation of the Poisson problem. An important difference between the variational formulations of the Poisson and the Darcy problems is that the Dirichlet boundary conditions that are enforced by modifying the trial space in the Poisson case are now natural, i.e., they appear as an integral in the right hand side, and the Neumann boundary conditions that are before natural, have to be enforced as essential boundary conditions in the Darcy case.  

The weak imposition of Dirichlet boundary conditions for the Poisson problem is a quite well-understood matter; let us refer, for instance, to~\cite{MR2683379,MR2045519,MR1365557}. 
On the other hand, the problem of weakly imposing the Neumann boundary conditions for the Darcy problem has not been sufficiently explored in the literature, especially in the context of a fictitious domain method. We base our analysis on the classical $\bm H(\dive)$-conforming Raviart-Thomas discretization in simplicial and quadrilateral meshes and adopt the Nitsche-type method introduced in~\cite{puppi_burman} for the weak imposition of the essential boundary conditions. The discrete formulation is very ill-posed because of the mismatch between the computational mesh and the physical domain where the PDEs live. We show that this affects not only the accuracy of the approximation scheme, but also the conditioning of the arising linear system. Our strategy to recover the well-posedness of the discrete formulation is in line with~\cite{MR2738930,MR2899249,MR3264337} and consists of adding to the variational formulation at the discrete level two weakly consistent ghost penalty operators acting separately on the velocity and on the pressure fields. The discrete functional setting is unusual since it is based on mesh-dependent norms scaling as $\bm H^1 \times H^1$, instead of the standard $\bm H(\dive)\times L^2$. Hence, we derive \emph{a priori} error estimates for the velocity and pressure fields which are optimal for the chosen topology, but not for the usual ones. A further drawback of our method is lost of the divergence-free property of the Raviart-Thomas element.

An outline of the remainder of the paper follows. We set the notation and introduce the model problem in Section~\ref{section1}. In Section~\ref{section2}, we introduce the model problem and its Raviart-Thomas discretization for both triangular and quadrilateral meshes. In Section~\ref{section3} we explain how we interpolate regular functions when the mesh does not fit the boundary of the physical domain. Section~\ref{section4} contains the discrete stabilized formulation and its numerical analysis: we rigorously derive the estimates guaranteeing the stability of our formulation and prove the \emph{a priori} error estimates. Section~\ref{section5} is devoted to the study of the condition number of the stiffness matrix. We prove that the ghost penalty stabilization restores the usual conditioning of the boundary-fitted case. In Section~\ref{section6} we explain how to deal and what changes in the case of pure Dirichlet boundary conditions. Finally, in Section~\ref{section7} we present some numerical experiments illustrating the theory.

\input{Sections/section1.tex}
\input{Sections/numerical_experiments.tex}

\appendix
\input{Sections/appendix.tex}
...
\clearpage
\bibliographystyle{plain}
\bibliography{bibliography}
\end{document}


\newcommand{\logLogSlopeTriangleOne}[5]
{

    \pgfplotsextra
    {
        \pgfkeysgetvalue{/pgfplots/xmin}{\xmin}
        \pgfkeysgetvalue{/pgfplots/xmax}{\xmax}
        \pgfkeysgetvalue{/pgfplots/ymin}{\ymin}
        \pgfkeysgetvalue{/pgfplots/ymax}{\ymax}

        \pgfmathsetmacro{\xArel}{#1}
        \pgfmathsetmacro{\yArel}{#3}
        \pgfmathsetmacro{\xBrel}{#1-#2}
        \pgfmathsetmacro{\yBrel}{\yArel}
        \pgfmathsetmacro{\xCrel}{\xArel}

        \pgfmathsetmacro{\lnxB}{\xmin*(1-(#1-#2))+\xmax*(#1-#2)} 
        \pgfmathsetmacro{\lnxA}{\xmin*(1-#1)+\xmax*#1} 
        \pgfmathsetmacro{\lnyA}{\ymin*(1-#3)+\ymax*#3} 
        \pgfmathsetmacro{\lnyC}{\lnyA+#4*(\lnxA-\lnxB)}
        \pgfmathsetmacro{\yCrel}{\lnyC-\ymin)/(\ymax-\ymin)} 

        \coordinate (A) at (rel axis cs:\xArel,\yArel);
        \coordinate (B) at (rel axis cs:\xBrel,\yBrel);
        \coordinate (C) at (rel axis cs:\xCrel,\yCrel);

        \draw[#5]   (A)-- node[pos=0.5,anchor=north] {1} 
                    (B)-- 
                    (C)-- node[pos=0.5,anchor=east] {#4} 
                    cycle;
    }
}


\begin{tikzpicture}

\definecolor{color1}{rgb}{255, 0, 0}
\definecolor{color2}{rgb}{0,255,0}
\definecolor{color3}{rgb}{0,0,255}
\definecolor{color4}{rgb}{255,0,255}
\definecolor{color5}{rgb}{1,0.9,0}
\definecolor{color6}{rgb}{0,255,180}
\definecolor{color7}{rgb}{0,70,0}

  \begin{loglogaxis}[
        scale only axis, 
        xlabel={$h$},
        ylabel={inf-sup constant},
      xtick=data,
      xticklabels={$1/2$, $1/4$, $1/8$, $1/16$, $1/32$},
        legend pos=north east,
legend style={font=\Large},
grid=major,
xlabel style={font=\Large},
ylabel style={font=\Large},
xticklabel style={font=\Large},
yticklabel style={font=\Large}
      ]

\addplot[semithick, color1, mark=o, mark size=3, mark options={draw=color1}] table[ x index = 0, y index  =1] {\dataDir/Darcy/inf_sup_b1_pentagon.txt};
\addlegendentry{$\varepsilon=10^{-3}$};
\addplot[semithick, color2, mark=square, mark size=3, mark options={draw=color2}] table[ x index = 0, y index = 2] {\dataDir/Darcy/inf_sup_b1_pentagon.txt};
\addlegendentry{$\varepsilon=10^{-5}$};
\addplot[semithick, color3, mark=*, mark size=3, mark options={draw=color3}] table[ x index = 0, y index = 3] {\dataDir/Darcy/inf_sup_b1_pentagon.txt};
\addlegendentry{$\varepsilon=10^{-7}$};
\addplot[semithick, color4, mark=diamond, mark size=3, mark options={draw=color4}] table[ x index = 0, y index = 4] {\dataDir/Darcy/inf_sup_b1_pentagon.txt};
\addlegendentry{$\varepsilon=10^{-9}$};
\addplot[semithick, color5, mark=otimes, mark size=3, mark options={draw=color5}] table[ x index = 0, y index = 5] {\dataDir/Darcy/inf_sup_b1_pentagon.txt};
\addlegendentry{$\varepsilon=10^{-11}$};
\addplot[semithick, color6, mark=triangle, mark size=3, mark options={draw=color6}] table[ x index = 0, y index = 6] {\dataDir/Darcy/inf_sup_b1_pentagon.txt};
\addlegendentry{$\varepsilon=10^{-13}$};
\end{loglogaxis}

\end{tikzpicture}

%% file: Sections/section1.tex
\section{Notation and model problem} \label{section1}
We briefly introduce some useful notations and definitions for the forthcoming analysis. Let $d\in\{2,3\}$ and $D$ be a Lipschitz-regular \emph{domain} (subset, open, bounded, connected) of $\mathbb R^d$, with boundary $\partial D$ and unit outer normal $\n$. Let $L^2(D)$ denote the space of square integrable functions on $D$, equipped with the usual norm $\norm{\cdot}_{L^2(D)}$. Let $L^2_0(D)$ be the subspace of $L^2(D)$ of functions with zero average, where the average of $v\in L^2(D)$ is $\overline v:=\frac{1}{\abs{D}}\int_D v$. For a given $\varphi:D\to \R$ sufficiently regular and $\bm\alpha$ multi-index with $\abs{\bm\alpha}:=\sum_{i=1}^d\alpha_i$, we define $\displaystyle D^{\bm\alpha}\varphi:={\frac{\partial^{\abs{\bm\alpha}}\varphi}{\partial x_1^{\alpha_1}\dots\partial x_d^{\alpha_d}}}$ and $\partial_n^j \varphi:=\sum_{\abs{\bm\alpha}=j}D^{\bm\alpha} \varphi  \n ^{\bm\alpha}$, where $\n^{\bm\alpha}:=n_1^{\alpha_1}\dots n_d^{\alpha_d}$, when $\varphi:\Omega\to \R$ is smooth enough. We indicate by $H^k(D)$, for $k\in\N$, the standard Sobolev space of functions in $L^2(D)$ whose $k$-th order weak derivatives belong to $L^2(D)$, equipped with the norm $\norm{\varphi}^2_{H^k(D)}:=\sum_{\abs{\bm \alpha}\le k} \norm{D^{\bm\alpha}\varphi}^2_{L^2(D)}$. Sobolev spaces of fractional order $H^r(D)$, $r\in\R$, can be defined by interpolation techniques, see~\cite{MR2424078}. The space $H^1_{0,\Sigma}(D)$ consists of functions in $H^1(D)$ with vanishing trace on $\Sigma$. We write $H^1_{0}(D)$ instead of $H^1_{0,\partial D}(D)$. Let us denote $\bm L^2(D):=\left(L^2(D)\right)^d$ and $\bm H^k(D):= \left( H^k(D) \right)^d$. We define the Hilbert space $\bm H(\dive;D)$ of vector fields in $\bm L^2(D)$ with divergence in $L^2(D)$, endowed with the graph norm, denotes as $\norm{\cdot}_{H(\dive;D)}$. Moreover, we set $\bm H_{0,\Sigma}(\dive;D):=\{\vv\in \bm H(\dive;D): \vv\cdot\n =0\;\text{on}\;\Sigma\}$ and $\bm H_{0}(\dive;D):=\bm H_{0,\partial D}(\dive;D)$. Let $H^\frac{1}{2}(\partial D)$ be the range of the trace operator of functions in $H^1(D)$ and we define its restriction to $\Sigma$ as $H^{\frac{1}{2}}(\Sigma)$. Both $H^\frac{1}{2}(\partial D)$ and $H^{\frac{1}{2}}(\Sigma)$ can be endowed with an intrinsic norm, see~\cite{MR2328004}. The dual space of $H^{\frac{1}{2}}(\Sigma)$ is denoted $H^{-\frac{1}{2}} (\Sigma)$. The duality pairing between $H^{\frac{1}{2}}(\Sigma)$ and $H^{-\frac{1}{2}}(\Sigma)$ will be denoted with a formal integral notation.  Finally, let $\bm H^\frac{1}{2}(\partial D):= \left( H^\frac{1}{2}(\partial D) \right)^d$, $\bm H^{\frac{1}{2}}(\Sigma):= \left( H^{\frac{1}{2}}(\Sigma) \right)^d$ and $\bm H^{-\frac{1}{2}} (\Sigma):= \left( H^{-\frac{1}{2}} (\Sigma) \right)^d$. We will denote as $\mathbb Q_{r,s,t}(D)$ the vector space of polynomials of degree at most $r$ in the first variable, at most $s$ in the second and at most $t$ in the third one (analogously for th case $d=2$) in $D$, $\mathbb P_u(D)$ the vector space of polynomials of degree at most $u$ in $D$, $\tilde{\mathbb P}_\ell(D)$ the vector space of homogeneous polynomials of degree $\ell$ in $D$. We may write $\mathbb Q_k(D)$ instead of $\mathbb Q_{k,k}(D)$ or $\mathbb Q_{k,k,k}(D)$. In the same way, we can define vector spaces of polynomials defined in hypersurfaces of $\R^d$. With a slight abuse of notation, we will use the same symbol $\abs{\cdot}$ to denote both the $d$-dimensional Lebesgue measure and the $(d-1)$-dimensional Hausdorff measure. Given $D\subset \R^d$ and $\Sigma$ a hypersurface of $\R^d$ or a subset of it, $\abs{D}$ and $\abs{\Sigma}$ denote the $d$-dimensional Lebesgue measure of $D$ and the $(d-1)$-dimensional Hausdorff measure of $\Sigma$, respectively. $E^\circ$ and $\operatorname{int} E$ denotes the interior of $E\subset \R^d$. Given $x,y\in\R$, we will write $x\lesssim y$ if there exists $c>0$, independent of $x$,$y$, such that $x\le c y$ and $x\sim y$ if $x\lesssim y$ and $y\lesssim x$. $\mathcal O$ will denote the classical Landau symbol. $C$ will denote generic positive constants that may change with each occurrence throughout the document but are always independent of the local mesh size and the mutual position of mesh and physical domain unless otherwise specified. 

Let $\Omega$ be a Lipschitz-regular domain of $\R^d$ with boundary $\Gamma$ such that $\Gamma = \Gamma_N \cup \Gamma_D$, where $\Gamma_N$, $\Gamma_D$ are non-empty, open, and disjoint. The Darcy problem is a linear system of partial differential equations modeling the flow of groundwater through a porus medium, here represented by $\Omega$, with permeability $\kappa$.
Given $\f\in\bm L^2(\Omega)$, $g\in L^2(\Omega)$, $u_N\in H^{-\frac{1}{2}}(\Gamma_N)$, $p_D\in H^{\frac{1}{2}}(\Gamma_D)$, we look for $\u:\Omega\to\R^d$ and $p:\Omega\to\R$ such that
\begin{equation}
\begin{aligned}\label{prob:cont}
		\kappa^{-1} \u -\nabla p = \f, \qquad & \text{in}\;\Omega,\\
		\dive\u = g,\qquad&\text{in}\;\Omega,\\
		\u\cdot\n=u_N,\qquad&\text{on}\;\Gamma_N,\\
		p=p_D,\qquad&\text{on}\;\Gamma_D.
\end{aligned}
\end{equation}
The unknowns $\u$ and $p$ represent, respectively, the seepage velocity and the pressure of the fluid.  The first equation of \eqref{prob:cont} is called \emph{Darcy law} relating the velocity and the pressure gradient of the fluid, the second one expresses \emph{mass conservation} (when $g\equiv 0$), the third and the fourth equations are, respectively, a \emph{Neumann boundary condition} for the velocity field and a \emph{Dirichlet boundary condition} for the pressure. Moreover, $\kappa\in\mathbb R^{d\times d}$ is symmetric positive definite with eigenvalues $\lambda_i$ such that $0<\lambda_{\min}\le\lambda_i\le\lambda_{\max}<+\infty$, for every $i=1,\dots,d$.
In the subsequent analysis we are going to consider, for the sake of simplicity, $\kappa=I$ the identity matrix.

\section{The finite element discretization} \label{section2}
Let us introduce $\left( \mathcal T_h\right)_{h>0}$, a family of admissible triangular or quadrilateral meshes (see, for instance, Chapter~3 of~\cite{MR1299729}) such that, for every $h>0$, $\overline\Omega\subsetneq\Omega_\T$, $\Omega_{\mathcal T}$ being the \emph{fictitious domain}, i.e., $\Omega_{\mathcal T}:=\operatorname{int}\cup_{K\in\mathcal{T}_h} K$. Let us denote the collection of all the facets (edges if $d=2$ and faces if $d=3$) as $\mathcal F_h$ and partition it into two disjoint sets: the faces lying on the boundary of $\Omega_{\mathcal T}$, denoted as $\mathcal F_h^\partial$, and $\mathcal F_h^i$, the internal ones. For every cut element $K\in\mathcal T_h$, let us denote its intersection with the boundary as $\Gamma_K$. It will be clear from the context if with $\Gamma_K$ we mean the intersection with the whole boundary, i.e., $\Gamma_K:=\Gamma\cap K$ or with just one of its disjoint components $\Gamma_N$, i.e., $\Gamma_K:=\Gamma_N\cap K$, or $\Gamma_D$, i.e., $\Gamma_K:=\Gamma_D\cap K$. It will also be useful to define the collection of the cut-elements, namely $\mathcal{G}_h:=\{K\in\mathcal{T}_h: \abs{\Gamma_K}\neq 0\}$, its two sub-collections $\mathcal G_h(\Gamma_N):=\{K\in\mathcal{T}_h: \abs{\Gamma_N \cap K}\neq 0\}$ and $\mathcal G_h(\Gamma_D):=\{K\in\mathcal{T}_h: \abs{\Gamma_D \cap K}\neq 0\}$, that we assume to be disjoint, and the interior part $\Omega_{I,h}:=\Omega\setminus \cup_{K\in\mathcal G_h} K$. Let $\T_h(\Omega_{I,h}):=\{K\in\T_h:K\subset\overline\Omega_{I,h}\}$. The collections of the internal and boundary faces entirely contained in the bulk of the domain are respectively denoted as $\mathcal F_h^i(\Omega_{I,h})$ and $\mathcal F_h^\partial(\Omega_{I,h})$. The collection of facets in the boundary region are denoted as $\mathcal F_h^{\Gamma}:=\{f \in\mathcal F_h^i: \exists\ K\in\mathcal G_h\ \text{such that}\ f\subset\partial K\}$.
 
 Let us  assign to each element $K\in\mathcal T_K$ its diameter $h_K$ and denote $h:=\max_{K\in\mathcal T_h}h_K$. We assume the background mesh to be \emph{shape-regular}, i.e., there exists $\sigma>0$, independent of $h$, such that $\max_{K\in\mathcal T_h}\frac{h_K}{\rho_K}\le \sigma$, $\rho_K$ being the diameter of the largest ball inscribed in $K$. Moreover, $\T_h$ is supposed to be \emph{quasi-uniform} in the sense that there exists $\tau>0$, independent of $h$, such that $\min_{K\in\mathcal T_h} h_K\ge \tau h$. We fix an orientation for the internal faces, i.e., given $f\in\mathcal F_h^i$ such that $f=\partial K_1 \cap\partial K_2$, we assume that the unit normal on $f$ points from $K_1$ toward $K_2$. Let $\varphi:\Omega\to\R$ be smooth enough. Then, for all $f\in\mathcal F_h^i$ and a.e. $x\in f$, we define the \emph{jump} of $\varphi$ as
 \begin{align*}
[\varphi](x):= \restr{\varphi}{K_1}(x) - \restr{\varphi}{K_2}(x), 
\end{align*}
where $f=\partial K_1\cap \partial K_2$. We may remove the subscript $f$ when it is clear from the context to which facet we refer to.

The following mild assumptions on how the mesh may be intersected by the boundary $\Gamma$ will be helpful. First, let us require that the number of facets to be crossed to move from a cut element $K$ to an uncut element $K'$ is uniformly bounded with respect to $h$.
\begin{assumption}\label{assumption:mesh}
	There exists $N>0$ such that, for every $h>0$ and $K\in\mathcal G_h$, there exist $K'\in\T_h(\Omega_{I,h})$ and at most $N$ elements $\left(K_i\right)_{i=1}^N\subset\T_h$ such that $K_1=K$, $K_N=K'$ and $K_i\cap K_{i+1}$ is a cut facet, for every $1\le i\le N-1$.
\end{assumption}	
Then, we assume that it is possible to subdivide the boundary region into a fixed (i.e., independent of $h$) number of patches such that every cut element belongs to one patch, every patch contains an uncut element, the diameter of every patch is $\mathcal O(h)$, and all the patches are diffeomorphic to a reference patch.
\begin{assumption}\label{assumption:patch}
The boundary zone $\bigcup_{K\in\mathcal G_h} K$ can be decomposed into $N_{\mathcal P}$ patches $\left(\mathcal P_\ell \right)_{\ell=1}^{N_{\mathcal P}}$, $\overline P_\ell=\bigcup_{ K\in\mathcal P_\ell } K$, $1\le \ell\le  N_{\mathcal P}$, satisfying:
\begin{enumerate}[(i)]
	\item\label{condition:i} for every $K\in\mathcal G_h$ there exists $1\le\ell\le N_{\mathcal P}$ such that $K\in \mathcal P_\ell$;
	\item\label{condition:ii} for every $1\le \ell\le  N_{\mathcal P}$ there exists $K'\in\mathcal T_h(\Omega_{I,h})\cap \mathcal P_\ell$;
	\item\label{condition:iii} there exists $C_h>0$ such that, for every $1\le \ell \le N_{\mathcal P}$ and $K\in \mathcal P_\ell$, it holds $h_K \ge C_h h_{\mathcal P_\ell}$, where $h_{\mathcal P_\ell}:= \operatorname{diam} \left(  P_\ell \right) $;
	\item for every $1\le \ell\le N_{\mathcal P}$ there exist an affine map $F_{\ell}$ and a reference patch $\hat P$ with $\operatorname{diam}(\hat P)=1$, such that $F_\ell:\hat P\to P_\ell$ is invertible. 
\end{enumerate}	
For every $1\le \ell\le  N_{\mathcal P}$, we denote as $\mathcal F_\ell:=\{f\in\mathcal F_h^i: f\subseteq P_\ell, f\not\subset \partial P_\ell\}$. 
\end{assumption}
Let $V_h\subset \bm H(\dive;\Omega_\T)$ and $Q_h\subset L^2(\Omega_\T)$ be the Raviart-Thomas finite element space defined in the whole fictitious domain $\Omega_\T$. For the sake of completeness, we recall its construction (for the boundary-fitted case the interested reader is referred, for instance, to~\cite{puppi_burman,MR2139400,MR2059447,MR1299729}).

 In the case of triangles and tetrahedra, we consider as reference element $\hat K$ is the unit $d$-simplex, i.e., the triangle of vertices $(0,0)$, $(1,0)$, $(0,1)$ if $d=2$ and the tetrahedron of vertices $(0,0,0)$, $(1,0,0)$, $(0,1,0)$, $(0,0,1)$ when $d=3$. For quadrilaterals, the reference element $\hat K$ is the unit $d$-cube $[0,1]^d$.

Let us construct the spaces of polynomials on the reference element. For the triangular meshes (see~\cite{MR3097958}), 
\begin{equation*}
	\mathbb{RT}_k(\hat K):=\left(\mathbb{P}_k(\hat K)\right)^d\oplus \bm x\tilde{\mathbb P}_k(\hat K),	\qquad \mathbb M_k(\hat K):= \mathbb P_k(\hat K),
\end{equation*}
while, in the case of quadrilaterals (see~\cite{MR2139400,MR3097958}), it reads as follows 
\begin{equation*}
	\mathbb{RT}_k(\hat K):=
	\begin{cases}
		\mathbb{Q}_{k+1,{k}}(\hat K)\times\mathbb{Q}_{k,k+1}(\hat K),\qquad&\text{if}\;d=2,\\
		\mathbb{Q}_{k+1,{k},k}(\hat K)\times \mathbb{Q}_{k,k+1,k}(\hat K)\times\mathbb{Q}_{k,k,k+1}(\hat K),\qquad&\text{if}\;d=3,
	\end{cases}
	\qquad \mathbb M_k(\hat K):= \mathbb Q_k(\hat K).
\end{equation*}
Let $\F_K:\hat K\to K$ be a diffeomorphism mapping the reference element to a general $K\in\mathcal T_h$. For triangles we consider an affine bijection $\F_K(\hat x):= B_K \hat x + b_K$, where $B_K\in \mathbb R^{d\times d}$ is non-singular and invertible, and $b_K\in\mathbb R^d$. For quadrilateral meshes, an affine mapping would constrain us to parallelograms, hence we let $\F_K$ being a bijection bilinear for each component, so that we can map the reference element to abritrary convex quadrilaterals. The diffeomorphism $\F_K$ induces the pull-back operators
\begin{alignat}{3}\label{preliminaries_fem:diffeomorphism}
	\F^p_K:& L^2(K) \to L^2(\hat K),\qquad &&\F^p_K(q):=q\circ \F_K, \\
	\F^v_K:& \bm H (\dive; K) \to \bm H (\dive; \hat K) ,\qquad &&\F^v_K(\vv):=\abs{\det(D\F_K)}D\F_K^{-1}\vv.
\end{alignat}
Let us observe that $\F^p_K$ and $\F^v_K$ are isometric isomorphisms (see, for instance,~\cite{MR2566587}). For the construction of our discrete spaces, we will use $\F^p_K$ and $\F^v_K$. The inverse of $\F^v_K$ is commonly known as the \emph{contravariant Piola transform} or, more simply, \emph{Piola transform}, and we denote it as
\begin{align*}
	\mathcal P_K : \bm H (\dive; \hat K) \to \bm H (\dive; K),\qquad \mathcal P_{K} (\hat \vv) :=\abs{\det\left(D\F_K  \right)}^{-1}D\F_K \hat\vv.
\end{align*}	
We define the following finite-dimensional vector spaces
\begin{equation*}
	\begin{aligned}
		V_h&:=\big\{\vv_h \in \bm H\left(\dive;\Omega\right):\F_K^v \left(\restr{\vv_h}{K}\right)\in \mathbb{RT}_k(\hat K),\quad\forall\ K\in\mathcal T_h\big\}\\
		& = \big\{\vv_h \in \bm H\left(\dive;\Omega\right):\restr{\vv_h}{K}\in \mathbb{RT}_k( K),\quad\forall\ K\in\mathcal T_h\big\}, \\
		Q_h&:=\big\{q_h\in L^2\left(\Omega\right):\F^p_K\left(\restr{q_h}{K}\right)\in\mathbb{M}_{k}\left(\hat K\right),\quad\forall\ K\in\mathcal T_h\big\} \\
		& = \big\{q_h\in L^2\left(\Omega\right): \restr{q_h}{K}\in\mathbb{M}_{k}\left( K\right),\quad\forall\ K\in\mathcal T_h\big\},
	\end{aligned}
\end{equation*}
where $\mathbb{RT}_k(K) :=\{\mathcal P_{K}\hat\w_h: \hat\w_h\in\mathbb {RT}_k(\hat K) \}$.  Remember that in the pure Neumann case, i.e., $\Gamma=\Gamma_N$, we have to filter out the constant functions from $Q_h$ by imposing the zero average constraint.

\section{Interpolation strategy}\label{section3}
Let us construct the interpolation operators by using the degrees of freedom of $V_h$ and $Q_h$. For every $\vv\in \bm H^s(\hat K)$, $s>\frac{1}{2}$, $r_{\hat K}$ is uniquely defined by:
\begin{equation}\label{dofs_velocity}
	\begin{aligned}
		\int_{\hat f} r_{\hat K}\vv\cdot \n \hat q_h = \int_{\hat f} \hat \vv\cdot\n \hat q_h, \qquad&\forall\ \hat q_h\in\Psi_k(\hat f), \\
		\int_{\hat K} r_{\hat K}\vv\cdot \w_h = \int_{\hat K} \hat \vv\cdot \hat \w_h, \qquad&\forall\ \hat \w_h\in \Psi_k\left(\hat K\right),\qquad\text{if}\;k>0,\\
	\end{aligned}
\end{equation}
where, for triangles,
\begin{equation*}
	\Psi_k\left(\hat K\right):=\left(\mathbb P_{k-1}(\hat K)\right)^d,\qquad \Psi_k(\hat f):=\mathbb P_{k}(\hat f),\\
\end{equation*}
and, for quadrilaterals,
\begin{align*}
	\Psi_k\left(\hat K\right):=&
	\begin{cases}
		\mathbb{Q}_{k-1,{k}}(\hat K)\times\mathbb{Q}_{k,k-1}(\hat K),\qquad&\text{if}\;d=2,\\
		\mathbb{Q}_{k-1,{k},k}(\hat K)\times \mathbb{Q}_{k,k-1,k}(\hat K)\times\mathbb{Q}_{k,k,k-1}(\hat K),\qquad&\text{if}\;d=3,
	\end{cases}\\
	\Psi_k\left(\hat f\right):=&
	\begin{cases}
		\mathbb P_k(\hat f),\qquad&\text{if}\;d=2,\\
		\mathbb Q_k(\hat f),\qquad&\text{if}\;d=3,
	\end{cases}
\end{align*}
for all facets (edges if $d=2$, faces if $d=3$) $\hat f$ of $\hat K$. We define $r_K: \bm H^s(K)\to \mathbb{RT}_k(K)$, so that $r_K = (\F_K^v)^{-1}\circ r_{\hat K}\circ \F_K^v = \mathcal P_K \circ r_{\hat K}\circ \mathcal P_K^{-1}$. The global interpolant $r_h^\T:\bm H(\dive;\Omega_\T)\cap \prod_{K\in\mathcal T_h}\bm H^s(K)  \to V_h$ is readily defined by gluing together the local interpolation operators, that is $\restr{r_h^\T}{K}:=r_K$, $K\in\T_h$. Let us move to the pressure case. We start from the reference element $\hat K$, and define $\Pi_{\hat K}:L^2(\hat K)\to\mathbb M_k(\hat K)$ which acts on $\xi \in L^2(\hat K)$ as
\begin{align}\label{dofs_pressure}
	\int_K \Pi_{\hat K} \xi q_h&= \int_{\hat K} \xi q_h, \qquad \forall\ q_h\in \mathbb M_k(\hat K).
\end{align}
Then, for a general $K\in\mathcal T_h$, we define $\Pi_{ K}:L^2(K)\to\mathbb M_k(\hat K)$ such that $\Pi_K = (\F_K^p)^{-1}\circ \pi_{\hat K}\circ \F_K^p $. Finally, let $\Pi_h^\T: L^2\left(\Omega_\T\right)\to Q_h$ such that for every $K\in\mathcal T_h$, $\restr{\Pi_h^\T}{K}:=\Pi_K$.
\begin{remark}
	Given $\vv\in \bm H(\dive;\Omega_\T)\cap \prod_{K\in\mathcal T_h}\bm H^s(K)$, the degrees of freedom for $r_{\hat K}\left(\restr{\vv}{\hat K}\right)$ given by~\eqref{dofs_velocity} are invariant under $\F_K^v$, for every $K\in\T_h$. Similarly, given $q\in L^2(\Omega_\T)$, the degrees of freedom for $\Pi_{\hat K}\left(\restr{q}{\hat K}\right)$ given by~\eqref{dofs_pressure} are invariant under $\F_K^p$, for every $K\in\T_h$.
\end{remark}
The key property of $r_h^\T:\bm H(\dive;\Omega_\T)\cap \prod_{K\in\mathcal T_h}\bm H^s(K) \to V_h$, $s>\frac{1}{2}$, and $\Pi_h^\T: L^2(\Omega_\T)\to Q_h$ is that the following diagram commutes and, in particular, $\dive V_h = Q_h$.
\begin{equation}\label{comm_diag}
	\begin{CD}
		\bm H(\dive;\Omega_\T)\cap \prod_{K\in\mathcal T_h}\bm H^s(K) @>\dive>> L^2(\Omega_{\mathcal T}) \\
		@VVr_h^\T V @VV\Pi_h^\T V \\
		V_h @>\dive>> Q_h.
	\end{CD}	
\end{equation}
From~\cite{MR2923416}, there exist $\bm E: \bm H^{t}\left(\Omega\right)\to\bm H^{t}\left(\R^d\right)$, $t\ge 1$, and $E:  H^{r}\left(\Omega\right)\to  H^{r}\left(\R^d\right)$, $r\ge 1$, universal (degree-independent) Sobolev-Stein extensions such that $\dive \circ \bm E = E\circ \dive E$. We define, for $t\ge 1$ and $r\ge 1$,
\begin{equation*}
	\begin{aligned}
		r_{h}:& \bm H^{t}\left(\Omega\right)\to V_h,\qquad \vv\mapsto r_h^\T \left( \restr{\bm E\left(\vv\right)}{\Omega_{\mathcal T}}    \right)  ,\\
		\Pi_{h}:& H^{r}\left(\Omega\right)\to Q_h,\qquad q\mapsto \Pi_{h}^\T \left( \restr{E\left(q\right)}{\Omega_{\mathcal T}}    \right).
	\end{aligned}
\end{equation*}
\begin{remark}\label{remark:mass_conservation}
	By construction, the commutativity of diagram~\eqref{comm_diag} is preserved when restricting to the physical domain $\Omega$, namely when employing $\restr{V_h}{\Omega}$, $\bm H(\dive;\Omega)$, $r_h$, and $\restr{Q_h}{\Omega}$, $L^2(\Omega)$, $\Pi_h$ instead of $V_h$, $\bm H(\dive;\Omega_{\mathcal T})$, $r_h^\T$, and $Q_h$, $L^2(\Omega_{\mathcal T})$, $\Pi_h^\T$. See Section~\ref{numerical_experiment:mass}.
\end{remark}



\section{The stabilized formulation}\label{section4}
Given $k\in\N$, the order of the Raviart-Thomas element employed for the discretization, we introduce two \emph{ghost penalty} jumps-based operators to enhance the stability of our discrete formulation and, in particular, to recover stability estimates independent of the mesh-boundary intersection (see~\cite{MR2738930,MR3264337,MR3802428,MR3268662}. We define
\begin{equation}\label{eq:ghost_penalty}
	\begin{aligned}
		\bm {j}_h (\w_h,\vv_h):=& \sum_{f\in \mathcal F_h^{\Gamma}} \sum_{j=0}^k h^{2j+1} \int_f [\partial_n^j \w_h] [\partial_n^j \vv_h],\qquad \w_h,\vv_h \in V_h,\\
		j_h (r_h,q_h):=&  \sum_{f\in \mathcal F_h^{\Gamma}} \sum_{j=0}^k h^{2j-1} \int_f [\partial_n^j  r_h] [\partial_n^j  q_h],\qquad r_h,q_h \in Q_h.	
	\end{aligned}
\end{equation}	
\begin{remark}\label{rmk:ghost_penalty}
	A wide zoo of ghost penalty operators has been proposed in the literature. For instance, it is possible to show that $\bm{j}_h(\cdot,\cdot)$ and ${j}_h(\cdot,\cdot)$ are equivalent to the following operators.
	\begin{align*}
		\bm{s}_h (\w_h,\vv_h):=& \sum_{\ell = 1}^{N_{\mathcal P}}  \int_{P_\ell}\left( \w_h -\bm {\pi}_\ell \w_h\right)\vv_h,   \qquad \w_h,\vv_h \in V_h,\\	
		\bm{g}_h (\w_h,\vv_h):=& \sum_{\ell = 1}^{N_{\mathcal P}} \int_{P_\ell}\left( \w_h -\bm {\mathcal E}_\ell \w_h\right)\vv_h,   \qquad \w_h,\vv_h \in V_h,\\	
		s_h (r_h,q_h):=&  \sum_{\ell =1}^{N_{\mathcal P}} h^{-2} \int_{P_\ell}\left( r_h -\pi_\ell r_h\right)q_h,  \qquad r_h,q_h \in Q_h,\\
		g_h (r_h,q_h):=&  \sum_{\ell =1}^{N_{\mathcal P}} h^{-2} \int_{P_\ell}\left( r_h -\mathcal E_\ell r_h\right)q_h,  \qquad r_h,q_h \in Q_h,
	\end{align*}
	namely it holds, respectively,
	\begin{alignat*}{3}
		\bm{j}_h(\vv_h,\vv_h)\lesssim& \bm{s}_h (\vv_h,\vv_h) \lesssim \bm{j}_h(\vv_h,\vv_h),\qquad 	  &&\bm{j}_h(\vv_h,\vv_h)\lesssim \bm{g}_h (\vv_h,\vv_h) \lesssim \bm{j}_h(\vv_h,\vv_h),\qquad\forall\ \vv_h \in V_h, \\
		j_h(q_h,q_h)\lesssim& s_h (q_h,q_h) \lesssim j_h(q_h,q_h), 	  &&j_h(q_h,q_h)\lesssim g_h (q_h,q_h) \lesssim j_h(q_h,q_h),\;\;\;\;\, \qquad\forall\ q_h\in Q_h.
	\end{alignat*}
	Here,
	\begin{align*}
		\bm{\pi}_{\ell}: \bm {L}^2(\mathcal P_\ell) \to \mathbb {RT}_k(\mathcal P_\ell),\qquad	\pi_{\ell}:  L^2(\mathcal P_\ell) \to \mathbb {M}_k(\mathcal P_\ell),
	\end{align*}
	denote the $L^2$-orthogonal projections onto the respective  finite-dimensional spaces, and
	\begin{align*}
		\bm {\mathcal E}_\ell: \mathbb{RT}_k(K') \to \mathbb {RT}_k(\mathcal P_\ell),\qquad		\mathcal E_{\ell}:  \mathbb{P}_k(K') \to \mathbb {M}_k(\mathcal P_\ell),
	\end{align*}
	are the canonical extensions of the respective polynomials, see~\cite{MR3942178,preuss_thesis}, where $K'$ is the uncut element of the $\ell$-th patch (see Assumption~\ref{assumption:patch}). From the implementation point of view, the operators $\bm{s}_h(\cdot,\cdot)$, ${s}_h(\cdot,\cdot)$ and $\bm{g}_h(\cdot,\cdot)$, ${g}_h(\cdot,\cdot)$  turn out to be a more convenient choice of $\bm{j}_h(\cdot,\cdot)$, $j_h(\cdot,\cdot)$, respectively, when a higher-order discretization is employed because of the evaluation of the high order derivatives. In the numerical experiments of Section~\ref{section7} we use the projection-based operators $\bm s_h(\cdot,\cdot)$ and $s_h(\cdot,\cdot)$. 
\end{remark}
We are now ready to introduce our stabilized discrete formulation. The idea is to employ the Nitsche formulation for the Darcy flow, which has been proposed and analyzed in~\cite{puppi_burman}, stabilizing it with the ghost penalty operators introduced above.

Find $\left( \u_h,p_h\right) \in V_h\times Q_h$ such that
\begin{equation}\label{cutfem:disc_pb}
	\begin{aligned}	
		a_h(\u_h,\vv_h)+ \bm{j}_h(\u_h,\vv_h) +  b_1 (\vv_h,p_h)  = \int_\Omega \f\cdot \vv_h + \int_{\Gamma_D} p_D \vv_h\cdot\n + h^{-1} \int_{\Gamma_N} u_N  \vv_h\cdot\n, \quad &\forall\ \vv_h\in V_h, \\
		b_1 (\u_h,q_h) - j_h(p_h,q_h) = \int_\Omega g q_h- \int_{\Gamma_N} u_N q_h,\quad & \forall\ q_h\in Q_h,
	\end{aligned}
\end{equation}
where
\begin{alignat*}{3}
	a_h(\w_h,\vv_h):=&\int_\Omega \w_h \cdot \vv_h + h^{-1}\int_{\Gamma_N}\w_h\cdot \n \vv_h\cdot \n,\qquad&& \w_h,\vv_h \in V_h, \\
	b_1(\vv_h,q_h):=&\int_\Omega q_h \dive \vv_h - \int_{\Gamma_N} q_h  \vv_h\cdot \n,\qquad &&\vv_h \in V_h, q_h\in Q_h.
\end{alignat*}
It will be convenient to rewrite~\eqref{cutfem:disc_pb} in the following more compact form.

Find $\left( \u_h,p_h\right) \in V_h\times Q_h$ such that
\begin{align}\label{cutfem:disc_pb2}
	\mathcal A_h \left( ( \u_h,p_h); (\vv_h,q_h) \right) = \mathcal F_h (\vv_h,q_h), \qquad\forall\ (\vv_h,q_h) \in V_h\times Q_h,
\end{align}
where, for $\left(\w_h,r_h\right),\left(\vv_h,q_h\right)\in V_h\times Q_h$,
\begin{align*}
	\mathcal A_h \left( (\w_h,r_h); (\vv_h,q_h) \right) : = & 	a_h(\w_h,\vv_h)+ \bm j_h(\w_h,\vv_h) +  b_1 (\vv_h,r_h) +  b_1 (\w_h,q_h) - j_h(r_h,q_h),\\
	\mathcal F_h (\vv_h,q_h) : = & \int_\Omega \f\cdot\vv_h + \int_{\Gamma_D} p_D \vv_h\cdot\n+ h^{-1} \int_{\Gamma_N} u_N \vv_h\cdot\n + \int_\Omega g q_h - \int_{\Gamma_N} u_N q_h.
\end{align*}
\begin{proposition}[Weak Galerkin Orthogonality]\label{prop:weak_go}
	Let $\left(\u,p\right) \in \bm H (\dive;\Omega)\times L^2(\Omega)$ be the solution of~\eqref{prob:cont} and $\left(\u_h,p_h\right)$ the one of~\eqref{cutfem:disc_pb2}. Then,
	\begin{align*}
		\mathcal A_h \left((\u-\u_h,p-p_h) ;(\vv_h,q_h) \right) = \bm j_h(\u,\vv_h) - j_h(p,q_h), \qquad\forall\ (\vv_h,q_h)\in V_h\times Q_h.
	\end{align*}
\end{proposition}
\begin{proof}
	The proof is trivial, hence we skip it.	
\end{proof}	
\begin{remark}
	We note that the formulation~\eqref{cutfem:disc_pb2} no longer fits into the framework of saddle-point problems. Hence, in order to study its stability, we will need to resort to the more general \emph{Banach-Ne\v cas-Babu\v ska Theorem}~\cite{MR2223503}. The case of pure Dirichlet boundary conditions will be covered separately in Section~\ref{section6}. Moreover, note that all the dimensionless parameters have been set for simplicity to $1$, unlike for the standard Nitsche method for the Poisson problem~\cite{MR2045519}, where the dimensionless parameter needs to be taken large enough.
\end{remark}	
We endow $V_h$ and $Q_h$ with the following mesh-dependent norms. 
\begin{alignat*}{3}
	\norm{\vv_h}^2_{0,h,\Omega_{\mathcal T}}: =&\norm{\vv_h}^2_{L^2(\Omega_{\mathcal T})} + \sum_{K\in\mathcal G_h(\Gamma_N)} h^{-1}\norm{\vv_h\cdot\n}^2_{L^2(\Gamma_K)},\qquad&& \vv_h\in V_h,\\%
	\norm{q_h}^2_{1,h,\Omega_{I}}: =& \sum_{K\in\mathcal T_h(\Omega_{I,h})}\norm{\nabla q_h}^2_{L^2(K)} + \sum_{f\in\mathcal F_h^i(\Omega_{I,h})} h^{-1}\norm{[q_h]}^2_{L^2(f)},\qquad&& q_h\in Q_h,\\
	\norm{q_h}^2_{1,h,\Omega_{\mathcal T}}:=& \sum_{K\in\mathcal T_h}\norm{\nabla q_h}^2_{L^2(K)} + \sum_{f\in\mathcal F_h^i} h^{-1}\norm{[q_h]}^2_{L^2(f\cap\Omega)} + \sum_{K\in\mathcal G_h(\Gamma_D)} h^{-1}\norm{q_h}^2_{L^2(\Gamma_K)}\quad&& q_h\in Q_h.
\end{alignat*}
The space $V_h\times Q_h$ is equipped with the product norm
\begin{align}\label{triple:norm}
	\vertiii{(\vv_h,q_h)}^2: = \norm{\vv_h}^2_{0,h,\Omega_{\mathcal T}} + \norm{q_h}^2_{1,h,\Omega_{\mathcal T}},\qquad (\vv_h,q_h)\in V_h\times Q_h.
\end{align}
Let us illustrate the salient properties of the ghost penalty operators that are needed to study the well-posedness of formulation~\eqref{cutfem:disc_pb}.
\begin{lemma}\label{semi_cs}
	The bilinear forms $\bm j_h(\cdot,\cdot)$ and $j_h(\cdot,\cdot)$ induce semi-inner products on $V_h$ and $Q_h$, respectively. In particular,	
	\begin{alignat*}{3}
		\bm j_h(\w_h,\vv_h) \le &	\bm j_h(\w_h,\w_h)^{\frac{1}{2}} \bm j_h(\vv_h,\vv_h)^{\frac{1}{2}},  \qquad&&\forall\ \w_h,\vv_h\in V_h,\\
		j_h(r_h,q_h) \le &	j_h(r_h,r_h)^{\frac{1}{2}} j_h(q_h,q_h)^{\frac{1}{2}}, \qquad&&\forall\ r_h,q_h\in Q_h.
	\end{alignat*}	
\end{lemma}
\begin{proof}
	It suffices to apply the Cauchy-Schwarz inequality, first in the $L^2$-setting and then in the $\ell^2$-setting. Let us show, for instance, the bound for $j_h(\cdot,\cdot)$. Given $r_h,q_h\in Q_h$, it holds
	\begin{align*}
		j_h(r_h,q_h) =&\sum_{f\in\mathcal F_h^{\Gamma}} \sum_{j=0}^k h^{2j-1} \int_f [\partial_n^j r_h] [\partial_n^j q_h] \le \sum_{f\in\mathcal F_h^{\Gamma}} \sum_{j=0}^k h^{\frac{1}{2}(2j-1)} \norm{[\partial_n^j r_h]}_{L^2(f)}h^{\frac{1}{2}(2j-1)} \norm{[\partial_n^j q_h]}_{L^2(f)} \\
		\le & \left(\sum_{f\in\mathcal F_h^{\Gamma}} \sum_{j=0}^k h^{2j-1} \norm{[\partial_n^j r_h]}_{L^2(f)}^2 \right)^{\frac{1}{2}}  \left(\sum_{f\in\mathcal F_h^{\Gamma}} \sum_{j=0}^k h^{2j-1} \norm{[\partial_n^j q_h]}_{L^2(f)}^2 \right)^{\frac{1}{2}}\\
		& = 	j_h(r_h,r_h)^{\frac{1}{2}} j_h(q_h,q_h)^{\frac{1}{2}}.
	\end{align*}
	The inequality for $\bm j_h(\cdot,\cdot)$ follows in a similar fashion.
\end{proof}
\begin{lemma}\label{lemma:massing}
	Let $K_1, K_2 \in \mathcal T_h$ with $f= \partial K_1\cap \partial K_2$. Let $\varphi_h$ be a piecewise polynomial such that $\varphi_1:=\restr{\varphi_h}{K_1} \in\mathbb M_{k_1}(K_1)$ and $\varphi_2:=\restr{\varphi_h}{K_2}\in\mathbb M_{k_2}(K_2)$, and let $k:=\max\{k_1,k_2\}$. There exist $C_1,C_2>0$, independent of $h>0$, but dependent on the shape-regularity constant and on $k$, such that
	\begin{equation}\label{eq:esperanto1} \\
		\begin{aligned}
			\norm{\varphi_1}^2_{L^2(K_1)} \le C_1&\left( \norm{\varphi_2}^2_{L^2(K_2)} + \sum_{j=0}^k h^{2j+1} \norm{ [\partial_n^j \varphi_h]}^2_{L^2(f)} \right),\\
			\norm{\frac{\partial\varphi_1}{\partial x_j}}^2_{L^2(K_1)} \le C_2&\left( \norm{\frac{\partial\varphi_2}{\partial x_j}}^2_{L^2(K_2)} + \sum_{j=0}^k h^{2j-1} \norm{ [\partial_n^j \varphi_h]}^2_{L^2(f)} \right),\qquad \forall\ 1\le j \le d.
		\end{aligned}
	\end{equation}
\end{lemma}
\begin{proof}
	The first inequality in~\eqref{eq:esperanto1} has been proven in Lemma~5.1 in~\cite{MR3268662}. For the second inequality, see Lemma~5.2 of~\cite{MR3942178}.
\end{proof}	

The following results enable us to control the norms in the whole $\Omega_\T$ in terms of the norms in the domain $\Omega_{I,h}$ through the ghost penalty operators.
\begin{theorem}\label{thm:massing}
	The following inequalities hold.
	\begin{alignat}{3}
		\norm{\vv_h}^2_{L^2(\Omega_\T)} \lesssim & \norm{\vv_h}^2_{L^2(\Omega_{I,h})} +  \bm j_h(\vv_h,\vv_h),\qquad&&\forall\ \vv_h\in V_h,\label{eq:polizei_velocity} \\
		\sum_{K\in\T_h}\norm{\nabla q_h}^2_{L^2(K)} \lesssim &	\sum_{K\in\T_h(\Omega_{I,h})}\norm{\nabla q_h}^2_{L^2(K)}  +  j_h(q_h,q_h),\qquad&&\forall\ q_h\in Q_h,\label{eq:polizei_pressures} \\
		\sum_{f\in \mathcal F_h^i}h^{-1}\norm{[q_h]}^2_{L^2(f)} \lesssim & 	\sum_{f\in\mathcal F_h^i(\Omega_{I,h})}h^{-1}\norm{[q_h]}^2_{L^2(f)}  +  j_h(q_h,q_h),\qquad&&\forall\ q_h\in Q_h.\label{eq:polizei}
	\end{alignat}	
\end{theorem}	
\begin{proof}
	Let us start with the proof of~\eqref{eq:polizei_velocity}. Since we can decompose $\Omega_\T = \Omega_{I,h} \cup \bigcup_{K\in\G_h} K$, it is sufficient to show
	\begin{align*}
		\sum_{K\in\G_h} \norm{\vv_h}^2_{L^2(K)}	\lesssim \norm{\vv_h}^2_{L^2(\Omega_{I,h})} + \bm j_h (\vv_h,\vv_h),\qquad&\forall\ \vv_h\in V_h,
	\end{align*}	
	which holds by Assumption~\ref{assumption:mesh}, the shape-regularity of $\T_h$, and Lemma~\ref{lemma:massing}. See also Lemma~2 of~\cite{MR3802428} and Proposition~5.1 of~\cite{MR3268662}. Inequality~\eqref{eq:polizei_pressures} for the pressures follows in a similar fashion. Let us move to~\eqref{eq:polizei}. Note that in view of Remark~\ref{rmk:ghost_penalty}, we may replace $j_h(\cdot,\cdot)$ with $s_h(\cdot,\cdot)$. Without loss of generality, let $f\in\mathcal F_h^i\setminus \mathcal F_h^i(\Omega_{I,h})$. Assumption~\ref{assumption:patch} guarantees the existence of a patch $\mathcal P_\ell$ such that $f\in \mathcal F_\ell$ and of an internal face $f'\in\mathcal F_\ell\cap \mathcal F_h^i(\Omega_{I,h})$. Let us take $q_h\in Q_h$ and map $P_\ell$ to the reference patch $\hat P$. We denote $\hat f:= F_\ell^{-1}(f)$, $\hat f':= F_\ell^{-1}(f')$, $\hat q_h:= q_h\circ F_\ell$, and $\hat \pi: L^2(\hat P)\to \mathbb M_k(\hat P)$ the $L^2$-orthogonal projection. Moreover, let $\hat {\mathcal F}:=\{f\in\mathcal F_h^i: f\subseteq \hat P , f\not\subset \partial P \}$. We also write
	\begin{align*}
		\hat s (\hat q_h, \hat q_h) := \int_{\hat P}\left(\hat q_h -\hat \pi \hat q_h\right) \hat q_h.
	\end{align*}
	It is sufficient to show that
	\begin{align*} 
		\norm{[\hat q_h]}^2_{L^2(\hat f)} \lesssim \norm{[\hat q_h]}^2_{L^2(\hat f')} + \hat s(\hat q_h,\hat q_h).
	\end{align*}	
	We decompose $\hat q_h = \hat q_1 + \hat q_2$, where $\hat q_1 \in \ker \hat s$ and $\hat q_2 \in \left(\ker \hat s \right)^\perp$, where the orthogonal complement is taken with respect to the $L^2$-scalar product on $\hat P$. Note that $\ker \hat s = \mathbb M_k (\hat P)$. We have, of course,
	\begin{align}\label{stereolab}
		\norm{[\hat q_1+\hat q_2]}^2_{L^2(f)} \le 	\norm{[\hat q_1+\hat q_2]}^2_{L^2(f)} + \norm{[\hat q_2]}^2_{L^2(f')}.
	\end{align}
	From norms equivalence on discrete spaces, it holds
	\begin{align}\label{otis}
		\sum_{f\in \hat{\mathcal F}} \norm{[\hat q_2]}^2_{L^2(f)} \lesssim \norm{\hat q_2}^2_{L^2(\hat P)}.	
	\end{align}	
	Indeed, it is easy to check that both terms in~\eqref{otis} are norms on $\left(\ker \hat s\right)^\perp$. In particular,~\eqref{otis} entails
	\begin{align}\label{otis2}
		\norm{[\hat q_2]}^2_{L^2(f)} \lesssim \norm{\hat q_2}^2_{L^2(\hat P)},\qquad\forall\ f\in\hat{\mathcal F}.	
	\end{align}
	By combining~\eqref{stereolab} and~\eqref{otis2}, we have $\norm{[\hat q_1+\hat q_2]}^2_{L^2(f)} \le \norm{[\hat q_1+\hat q_2]}^2_{L^2(f)} + \norm{\hat q_2}^2_{L^2(\hat P)}$.
	The reader can easily check $\hat s(\hat q_1+\hat q_2,\hat q_1 + \hat q_2) = \norm{\hat q_2}^2_{L^2(P)}$.
	Hence,
	\begin{align*}
		\norm{[\hat q_1+\hat q_2]}^2_{L^2(f)} \le \norm{[\hat q_1+\hat q_2]}^2_{L^2(f)} + \hat s(\hat q_1+\hat q_2,\hat q_1 + \hat q_2).
	\end{align*}
	The claim follows by scaling back to the physical patch $P_\ell$, summing over all the patches, using Assumption~\ref{assumption:patch}, and the shape-regularity of the mesh.
\end{proof}

\subsection{Stability estimates}
Let us prove the main ingredients that allow us to show the well-posedness of formulation~\eqref{cutfem:disc_pb2}.
\begin{proposition}\label{prop:continuity}
	The bilinear forms appearing in the weak formulation~\eqref{cutfem:disc_pb2} are continuous, namely, there exist $ M_{a},M_{b_1}, M_{\bm j},M_{j}>0$, such that
	\begin{alignat*}{3}
		\abs{a_h(\w_h,\vv_h)} \le & M_{a}  \norm{\w_h}_{0,h,\Omega_{\mathcal T}}\norm{\vv_h}_{0,h,\Omega_\T},\qquad&& \forall\ \w_h,\vv_h\in V_h, \\
		\abs{b_1(\vv_h,q_h)} \le & M_{b_1}  \norm{\vv_h}_{0,h,\Omega_{\mathcal T}}\norm{q_h}_{1,h,\Omega_\T},\qquad&& \forall\ \vv_h\in V_h,\forall\ q_h\in Q_h, \\
		\abs{\bm {j}_h(\w_h,\vv_h)}\le & M_{\bm j}  \norm{\w_h}_{0,h,\Omega_{\mathcal T}}\norm{\vv_h}_{0,h,\Omega_\T},\qquad &&\forall\ \w_h,\vv_h\in V_h, \\
		\abs{j_h(r_h,q_h)}\le & M_{j}  \norm{r_h}_{1,h,\Omega_{\mathcal T}}\norm{q_h}_{1,h,\Omega_\T},\qquad&& \forall\ r_h,q_h\in Q_h.
	\end{alignat*}
\end{proposition}
\begin{proof}
	Let us fix any $\w_h,\vv_h\in V_h$ and $r_h,q_h\in Q_h$. By Cauchy-Schwartz's inequality
	\begin{align*}
		\abs{a_h(\w_h,\vv_h)} \le & \norm{\w_h}_{L^2(\Omega)} \norm{\vv_h}_{L^2(\Omega)} + h^{-\frac{1}{2}}\norm{\w_h\cdot\n}_{L^2(\Gamma_N)} h^{-\frac{1}{2}}\norm{\vv_h\cdot\n}_{L^2(\Gamma_N)}\\
		\le & \norm{\w_h}_{0,h,\Omega_\T} \norm{\vv_h}_{0,h,\Omega_\T},
	\end{align*}	
	hence $M_a=1$. By integration by parts, we get
	\begin{align*}
		b_1(\vv_h,q_h) =& \int_\Omega q_h \dive \vv_h  - \int_{\Gamma_N} q_h \vv_h\cdot\n = \sum_{K\in\mathcal T_h} \int_{K\cap\Omega}q_h \dive \vv_h - \sum_{K\in\mathcal G_h(\Gamma_N)}\int_{\Gamma_K} q_h \vv_h\cdot\n\\
		=&-\sum_{K\in\mathcal T_h} \int_{K\cap\Omega}\nabla q_h\cdot \vv_h + \sum_{f\in\mathcal F_h^i}\int_{f\cap\Omega} \left[q_h \right]\vv_h\cdot\n + \sum_{K\in\mathcal G_h(\Gamma_D)}\int_{\Gamma_K}q_h \vv_h\cdot\n.
	\end{align*}
	From Cauchy-Schwarz's inequality, Lemma~\ref{lemma:disc_trace_ineq}, and a standard inverse estimate (Proposition~6.3.2 of~\cite{MR1299729}), we obtain
	\begin{align*}
		\sum_{f\in\mathcal F_h^i}\int_{f\cap\Omega} [q_h]\vv_h\cdot\n \le & \sum_{f\in\mathcal F_h^i} h^{-\frac{1}{2}} \norm{[q_h]}_{L^2(f\cap\Omega)} h^{\frac{1}{2}} \norm{\vv_h\cdot\n}_{L^2(f\cap\Omega)} \\
		\le &C\left(\sum_{f\in\mathcal F_h^i} h^{-1}\norm{[q_h]}^2_{L^2(f\cap\Omega)}\right)^{\frac{1}{2}}\left( \sum_{K\in\mathcal T_h} \norm{ \vv_h}^2_{L^2(K)}\right)^{\frac{1}{2}} \\
		\le & C \norm{q_h}_{1,h,\Omega_{\mathcal T}} \norm{\vv_h}_{0,h,\Omega_{\mathcal T}},
	\end{align*}
	and, analogously,
	\begin{align*}
		\sum_{K\in\mathcal G_h(\Gamma_D)}\int_{\Gamma_K} q_h\vv_h\cdot\n \le & \sum_{K\in\mathcal G_h(\Gamma_D)} h^{-\frac{1}{2}} \norm{q_h}_{L^2(\Gamma_K)} h^{\frac{1}{2}} \norm{\vv_h\cdot\n}_{L^2(\Gamma_K)} \\
		\le &C \left(\sum_{K\in\mathcal G_h(\Gamma_D)} h^{-1}\norm{q_h}^2_{L^2(\Gamma_K)}\right)^{\frac{1}{2}} \left(\sum_{K\in\mathcal T_h} \norm{ \vv_h}^2_{L^2(K)}\right)^{\frac{1}{2}} \\
		\le & C \norm{q_h}_{1,h,\Omega_{\mathcal T}} \norm{\vv_h}_{0,h,\Omega_{\mathcal T}}.
	\end{align*}
	Thus,
	\begin{align*}
		\abs{ b_1(\vv_h,q_h)}\le& \sum_{K\in\mathcal T_h}\norm{\nabla q_h}_{L^2(K\cap\Omega)}\norm{v_h}_{L^2(K\cap\Omega)} + C \norm{q_h}_{1,h,\Omega_{\mathcal T}} \norm{\vv_h}_{0,h,\Omega_{\mathcal T}}\\
		\le & C \norm{q_h}_{1,h,\Omega_{\mathcal T}} \norm{\vv_h}_{0,h,\Omega_{\mathcal T}} .
	\end{align*}
	The bounds for $\bm j_h(\cdot,\cdot)$ and $j_h(\cdot,\cdot)$ follow as well straightforward.
\end{proof}	
\begin{proposition}
	There exists $\alpha>0$, such that,
	\begin{align*}
		a_h(\vv_h,\vv_h) + \bm j_h(\vv_h,\vv_h) \ge \alpha \norm{\vv_h}^2_{0,h,\Omega_\T}, \qquad\forall\ \vv_h\in V_h.
	\end{align*}
\end{proposition}
\begin{proof}
	This is just a consequence of Theorem~\ref{thm:massing}.
\end{proof}
\begin{proposition}\label{prop:small_infsup}
	There exists $\beta_1>0$ such that
	\begin{align*}
		\inf_{q_h\in Q_h} \sup_{\vv_h\in V_h} \frac{b_1(\vv_h,q_h)}{\norm{\vv_h}_{0,h,\Omega_\T}} \ge \beta_1 \norm{q_h}_{1,h,\Omega_{I,h}}.
	\end{align*}
\end{proposition}	
\begin{proof}
	Let us fix $q_h\in Q_h$ and construct $\vv_h\in V_h$ using just the internal degrees of freedom of $V_h$, namely
	\begin{align*}
		\int_f \vv_h\cdot \n \varphi_h  = h^{-1 }\int_f [q_h] \varphi_h,\qquad&\forall\ f\in \mathcal F_h^i(\Omega_{I}), \forall\ \varphi_h\in\Psi_k(f) , \\
		\int_f \vv_h\cdot \n \varphi_h  = 0,\qquad&\forall\ f\in \mathcal F_h^\partial(\Omega_{I}), \forall\ \varphi_h\in \Psi_k(f),  \\
		\int_K \vv_h\cdot \bm \psi_h = -\int_K\nabla q_h\cdot\bm \psi_h, \qquad &\forall K\in\mathcal T_h(\Omega_{I}),\forall\ \bm\psi_h \in \Psi_k(K),  \qquad\text{if}\ k>0.
	\end{align*}
	We refer the reader to Section~\ref{section2} for the definitions of $\Psi_k(K)$ and $\Psi_k(f)$.	Let us extend $\vv_h$ to zero outside $\Omega_{I,h}$. It holds
	\begin{align*}
		b_1(\vv_h,q_h) =& -\sum_{K\in\mathcal T_h(\Omega_{I,h})} \int_K \nabla q_h\cdot \vv_h + \sum_{f\in\mathcal F_h^i(\Omega_{I,h})} \int_f [q_h]\vv_h\cdot\n\\
		=& \sum_{K\in\mathcal T_h(\Omega_{I,h})}\norm{\nabla q_h}^2_{L^2(K)} + \sum_{f\in\mathcal F_h^i(\Omega_{I,h})} h^{-1} \norm{[q_h]}^2_{L^2(f)} = \norm{q_h}^2_{1,h,\Omega_{I,h}}.
	\end{align*}
	Now, we prove that $\norm{\vv_h}_{0,h,\Omega_\T}\lesssim\norm{q_h}_{1,h,\Omega_{I,h}}$. By construction of $\vv_h$, it is sufficient to show
	\begin{align*}
		\sum_{K\in\mathcal T_h(\Omega_{I,h})} \norm{\vv_h}^2_{L^2(K)} \lesssim \sum_{K\in\mathcal T_h(\Omega_{I,h})}  \norm{\nabla q_h}^2_{L^2(K)} + \sum_{f\in\mathcal F_h^i(\Omega_{I,h})} h^{-1} \norm{[q_h]}^2_{L^2(f)}.
	\end{align*}	
	We mimic the proof of Proposition~2.1 in~\cite{MR3885757}.
	Let $K\in\T_h(\Omega_{I,h})$ and $\hat K$ be the the reference element. Let $f$ be a face of $K$ and $\hat f$ be its preimage through $\F_K$ (see Section~\ref{section2} for the definitions of $\hat K$ and $\F_K$). Finite dimensionality implies
	\begin{align*}
		\norm{\hat \vv_h}^2_{L^2(\hat K)} \lesssim \norm{\pi_{\hat K,k-1}\hat \vv_h}^2_{L^2(\hat K)} + \norm{\hat \vv_h\cdot\n}^2_{L^2(\hat f)},	
	\end{align*}	
	where $\pi_{\hat K,k-1}$ is the $L^2$-projection onto $\Psi_k(\hat K)$. By pushing forward to the element $K$, a scaling argument implies
	\begin{align*}
		\norm{ \vv_h}^2_{L^2( K)} \lesssim \norm{\pi_{ K,k-1} \vv_h}^2_{L^2( K)} + h\norm{ \vv_h\cdot\n}^2_{L^2( f)}.
	\end{align*}
	This time $\pi_{ K,k-1}$ is the $L^2$-projection onto $\Psi_k(K)$
	By construction of $\vv_h$,
	\begin{align*}
		\norm{ \vv_h}^2_{L^2( K)} \lesssim \norm{\pi_{ K,k-1}\nabla q_h}^2_{L^2( K)} + h^{-1}\norm{ \pi_{f,k} [q_h]}^2_{L^2( f)} 
		=  \norm{\nabla q_h}^2_{L^2( K)} + h^{-1}\norm{ [q_h]}^2_{L^2( f)},
	\end{align*}
	where $\pi_{f,k}$ denotes the $L^2$-projection onto $\Psi_k(f)$.
\end{proof}	

We are left with the proof of the well-posedness of formulation~\eqref{cutfem:disc_pb2}. In order to do that, we verify that the bilinear form $\mathcal A_h(\cdot;\cdot)$ satisfies the hypotheses of the Banach-Ne\v cas-Babu\v ska Theorem.
\begin{theorem}\label{thm:cutfem_stability}
	There exists $\eta>0$ such that	
	\begin{align*}
		\inf_{(\vv_h,q_h)\in V_h\times Q_h} \sup_{(\w_h,r_h)\in V_h\times Q_h} \frac{\mathcal A_h\left( (\vv_h,q_h)  ;(\w_h,r_h) \right)}{\vertiii{(\vv_h,q_h)} \vertiii{(\w_h,r_h)} }\ge\eta.
	\end{align*}
\end{theorem}
\begin{proof}
	Let $\left(\vv_h,q_h\right)\in V_h\times Q_h$ be arbitrary and $-\w_h$ be the element attaining the supremum in Proposition~\ref{prop:small_infsup}, namely
	\begin{align}
		-b_1(\w_h,q_h) = \norm{q_h}^2_{1,h,\Omega_{I,h}},\qquad
		\norm{\w_h}_{0,h,\Omega_\T} \lesssim \norm{q_h}_{1,h,\Omega_{I,h}}.\label{cutfem:eq4}
	\end{align}	
	We have
	\begin{align*}
		\mathcal A_h \left( (\vv_h,q_h);(-\w_h,0)\right) =& -a_h(\vv_h,\w_h) - \bm{j}_h(\vv_h,\w_h) - b_1(\w_h,q_h) + b_1(\vv_h,0) -j_h(q_h,0) \\
		\ge & -2 \norm{\vv_h}_{0,h,\Omega_\T} \norm{\w_h}_{0,h,\Omega_\T} + \norm{q_h}^2_{1,h,\Omega_{I,h}} + 0 + 0 \\
		\gtrsim & -2 \norm{\vv_h}_{0,h,\Omega_\T} \norm{q_h}_{1,h,\Omega_{I,h}}  + \norm{q_h}^2_{1,h,\Omega_{I,h}} \\
		\gtrsim &  -\frac{1}{\eps} \norm{\vv_h}^2_{0,h,\Omega_\T} + (1-\eps) \norm{q_h}^2_{1,h,\Omega_{I,h}},
	\end{align*}
	where $\eps>0$ arises from the Young inequality. On the other hand, it holds
	\begin{align*}
		\mathcal A_h \left( (\vv_h,q_h);(\vv_h,-q_h)\right) =& a_h(\vv_h,\vv_h) + \bm{j}_h(\vv_h,\vv_h) + b_1(\vv_h,q_h) - b_1(\vv_h,q_h) +j_h(q_h,q_h) \\
		\gtrsim & \norm{\vv_h}^2_{0,h,\Omega_\T} +j_h(q_h,q_h).
	\end{align*}
	By choosing $\left(\vv_h-\delta \w_h ,-q_h\right)\in V_h\times Q_h$, for $\delta>0$ to be set later on, we get
	\begin{align*}
		\mathcal A_h \left( (\vv_h,q_h);(\vv_h-\delta \w_h ,-q_h)\right) 	\gtrsim& \left(1- \frac{\delta}{\eps} \right) \norm{\vv_h}^2_{0,h,\Omega_\T} +\delta(1-\eps) \norm{q_h}^2_{1,h,\Omega_{I,h}} +j_h(q_h,q_h).
	\end{align*}	
	Let us take, for instance, $\eps = \frac{1}{2}$ and any $0<\delta<\frac{1}{2}$, so that
	\begin{align*}
		\mathcal A_h \left( (\vv_h,q_h);(\vv_h-\delta \w_h ,-q_h)\right) 	\gtrsim&  \norm{\vv_h}^2_{0,h,\Omega_\T} +\norm{q_h}^2_{1,h,\Omega_{I,h}} + j_h(q_h,q_h)\\
		\gtrsim &  \norm{\vv_h}^2_{0,h,\Omega_\T} +\norm{q_h}^2_{1,h,\Omega_\T},
	\end{align*}	
	where in the last inequality we used $\norm{q_h}^2_{1,h,\Omega_\T}\lesssim \norm{q_h}^2_{1,h,\Omega_{I,h}} + j_h(q_h,q_h)$, which follows from Theorem~\ref{thm:massing}. We are left with proving that $\vertiii{(\vv_h-\delta \w_h, q_h)} \lesssim \vertiii{(\vv_h,q_h)}$, which is a consequence of~\eqref{cutfem:eq4}.
\end{proof}
\subsection{\emph{A priori} error estimates}
\begin{theorem}\label{cutfem:thm_approx}
	There exists $C>0$ such that, for every $\left(\vv,q\right)\in  \bm H^{t}(\Omega)\times H^{r}(\Omega)$, $t\ge 1$, $r\ge 1$, it holds
	\begin{equation*}	
		\begin{aligned}	
			\norm{\bm E(\vv)- r_{h}\vv}_{0,h,\Omega_\T}+ \norm{ E(q)-\Pi_h q}_{1,h,\Omega_\T} \le C h^{s}\left(\norm{\vv}_{H^{t}(\Omega)}+\norm{q}_{H^{r}(\Omega)} \right),
		\end{aligned}	
	\end{equation*}	
	where $s:=\min\{t-1,r-1,k\}$, $\bm E$ and $E$ have been introduced in Section~\ref{section3}.
\end{theorem}	
\begin{proof}
	Let us start with the velocity. We have
	\begin{align*}
		\norm{\bm E(\vv) -  r_h \vv}^2_{0,h,\Omega_\T} = \norm{\bm E(\vv) -  r_h \vv}^2_{L^2(\Omega_\T)} + h^{-1} \sum_{K\in\mathcal G_h(\Gamma_N)} \norm{\left(\bm E(\vv)- r_h \vv \right)\cdot\n}^2_{L^2(\Gamma_K)}.
	\end{align*}
	For the volumetric term, let us apply the standard Deny-Lions argument (Chapter~3 of~\cite{MR1299729}) to the interpolant $r_h$ and get
	\begin{align*}
		\norm{\bm E(\vv) -  r_h \vv}^2_{L^2(\Omega_\T)}  = \sum_{K\in\mathcal T_h} \norm{\bm E(\vv) -  r_h \vv}^2_{L^2(K)} \le C \sum_{K\in\T_h}h^{2t} \norm{\bm E(\vv)}^2_{H^{t}(K)} \le C h^{2t} \norm{\vv}^2_{H^{t}(\Omega)}.
	\end{align*}	
	For the boundary part, let us first use the multiplicative trace inequality of Lemma~\ref{lemma:trace_ineq} (componentwise) and then, again, the Deny-Lions argument:
	\begin{align*}
		h^{-1}\sum_{K\in\mathcal G_h(\Gamma_N)} \norm{\left(\bm E(\vv) -  r_h \vv\right)\cdot\n}^2_{L^2(\Gamma_K)} \le &h^{-1}\sum_{K\in\mathcal G_h(\Gamma_N)} \norm{\bm E(\vv) -  r_h \vv}^2_{L^2(\Gamma_K)}\\
		\le& C h^{-1}\sum_{K\in\mathcal G_h(\Gamma_N)} \norm{\bm E(\vv) -  r_h \vv)}_{H^1(K)} \norm{\bm E(\vv)-  r_h \vv}_{L^2(K)} \\
		\le & C h^{-1}  \sum_{K\in\T_h} h^{t-1} \norm{\bm E(\vv)}_{H^{t}(K)} h^{t} \norm{\bm E(\vv)}_{H^{t}(K)}\\
		\le& C h^{2(t-1)} \norm{\vv}^2_{H^{t}(\Omega)}.
	\end{align*}
	We move to the pressure case.
	\begin{align}\label{cutfem:eq5}
		\norm{E(q)- \Pi_{h} q}^2_{1,h,\Omega_\T}=&	\sum_{K\in\mathcal T_h}\norm{\nabla\left(E(q)- \Pi_{h} q\right)}^2_{L^2(K)} + \sum_{f\in\mathcal F_h^i}h^{-1} \norm{\left[E(q)- \Pi_{h} q\right]}^2_{L^2(f\cap\Omega)}\\
		& +\sum_{K\in\G_h(\Gamma_D)}h^{-1}\norm{E(q)-\Pi_h q}^2_{L^2(\Gamma_K)}. 
	\end{align}	
	For the volumetric term one we may proceed as in the case of the velocity to easily obtain
	\begin{align*}
		\sum_{K\in\mathcal T_h}\norm{\nabla\left(E(q)- \Pi_{h} q\right)}^2_{L^2(K)} \le  C h^{2(r-1)}\norm{q}^2_{H^{r}(\Omega)}.
	\end{align*}
	Let us focus on the jump part of~\eqref{cutfem:eq5} and take $f\in\mathcal F_h^i$ such that $f=K_1\cap K_2$.
	Then, by Lemma~\ref{lemma:trace_ineq} and the Deny-Lions Lemma, we have
	\begin{align*}	
		h^{-1}\norm{\left[E(q)- \Pi_{h} q\right]}_{L^2(f\cap\Omega)}^2 = 	&
		h^{-1}	\norm{\left(\restr{\left(E(q)- \Pi_{h} q\right)}{K_1}-\restr{\left(E(q)- \Pi_{h} q\right)}{K_2}\right)}^2_{L^2(f)} \\
		\le &C h^{-1}\norm{E(q)- \Pi_{h} q}_{L^2(K_1)}\norm{E(q)- \Pi_{h} q}_{H^1(K_1)} \\
		&+Ch^{-1}\norm{E(q)- \Pi_{h} q}_{L^2(K_2)}\norm{E(q)-\Pi_{h} q}_{H^1(K_2)} \\
		\le & Ch^{2(r-1)}\left(\norm {E(q)}^2_{H^{r}(K_1)}+\norm{E(q)}^2_{H^{r}(K_2)} \right).
	\end{align*}
	Hence,
	\begin{align*}
		\sum_{f\in\mathcal F_h^i}h^{-1} \norm{\left[E(q)- \Pi_{h} q\right]}^2_{L^2(f\cap\Omega)} \le C h^{2(r-1)} \norm{q}^2_{H^{r}(\Omega)}.
	\end{align*}
	The bound for the boundary part of~\eqref{cutfem:eq5} follows in a similar fashion.
\end{proof}	
\begin{lemma}\label{lemma:bound_ghost}
	For every $\vv \in \bm H^{t}(\Omega)$ and $q\in H^r(\Omega)$, $t\ge 1$, $r\ge 1$,
	\begin{align*}
		\bm{j}_h(r_h \vv,r_h \vv)^{\frac{1}{2}} \lesssim  h^m\norm{\vv}_{H^{t}(\Omega)},\qquad
		j_h(\Pi_h q,\Pi_h q)^{\frac{1}{2}} \lesssim  h^\ell \norm{q}_{H^{r}(\Omega)},
	\end{align*}		
	where $m:=\min\{k+1,t\}$ and $\ell:=\min\{k,r-1\}$.
\end{lemma}	
\begin{proof}
	Note that
	\begin{equation}
		\begin{aligned}\label{cutfem:eq11}
			j_h(\pi_h q - E(q),\pi_h q-E(q)) = 	j_h(\Pi_h q ,\Pi_h q) - 	j_h(\Pi_h q ,E(q)) - 	j_h( E(q),\Pi_h q)+ 	j_h(E(q),E(q)).
		\end{aligned}
	\end{equation}
	Given $0\le k\le r-1$, for every $0\le j \le k$, $\bm\alpha$ multi-index such that $\abs{\bm\alpha}=j$, $D^{\bm\alpha} E(q)\in H^{r-j}(\Omega_\T)\subset  H^1(\Omega_\T)$, hence $[\partial^j_n E(q)]_f$ vanishes across every $f\in\mathcal F_h^i$. Thus,~\eqref{cutfem:eq11} implies that 
	\begin{align*}
		j_h(\Pi_h q, \Pi_h q)=j_h(\Pi_h q - E(q),\Pi_h q-E(q)).
	\end{align*}
	Now, by using Lemma~\ref{lemma:trace_ineq} and a standard approximation argument, we obtain
	\begin{align*}
		j_h(\Pi_h q - E(q),\Pi_h q-E(q)) = &\sum_{f\in\mathcal F_h^{\Gamma}} \sum_{i=0}^k h^{2i-1} \norm{[\partial^i_n(\Pi_h q-E(q))]}^2_{L^2(f)}\\
		\le &C \sum_{K\in\mathcal G_h} \sum_{i=0}^k h^{2i-1}  \norm{D^i \left( \Pi_h q-E(q)) \right)}_{L^2(K)}\norm{D^{i+1}\left(\Pi_h q-E(q))\right)}_{L^2(K)}\\
		\le& C h^{2\ell} \norm{E(q)}^2_{H^{r}(\Omega_\T)} \le C h^{2\ell} \norm{q}^2_{H^{r}(\Omega)},
	\end{align*}
	where $\ell:=\min\{k,r-1\}$ and the last inequality follows from the boundedness of $E$.
	The bound for $\bm{j}_h(\cdot,\cdot)$ follows in a completely similar fashion.
\end{proof}
\begin{theorem}\label{thm:apriori}
	Let $\left(\u,p\right)\in \bm H^{t}\left(\Omega\right)\times H^{r}\left(\Omega\right)$, $t\ge 1$, $r\ge 1$, be the solution of problem~\eqref{prob:cont}. Then, the finite element solution $\left(\u_h,p_h\right)\in  V_h\times  Q_h$ of~\eqref{cutfem:disc_pb} satisfies
	\begin{equation*}
		\norm{\u- \u_h}_{L^2(\Omega)}+ \left(\sum_{K\in \T_h}\norm{ \nabla \left(p- p_h\right)}^2_{L^2(K\cap\Omega)}\right)^{\frac{1}{2}}   \le C h^{s} \left(\norm{\u}_{H^{t}(\Omega)} + \norm{p}_{H^{r}(\Omega)}  \right),
	\end{equation*}
	where $s:= \min\{t-1,r-1,k\}$.
\end{theorem}
\begin{proof}
	Firstly, we observe that
	\begin{align*}
		\norm{\u- \u_h}_{L^2(\Omega)}+ \left(\sum_{K\in \T_h}\norm{ \nabla \left(p- p_h\right)}^2_{L^2(K\cap\Omega)}\right)^{\frac{1}{2}}  \le & \norm{\bm E(\u)-\u_h}_{0,h,\Omega_\T} + \norm{E(p)-p_h}_{1,h,\Omega_\T}\\
		\le &\sqrt{d} \vertiii{(\bm E(\u)-\u_h,E(p)-p_h)},
	\end{align*}
	it suffices to bound $\vertiii{(E(\u)-\u_h,E(p)-p_h)}$. Let us proceed by triangular inequality:
	\begin{align}\label{eq:covid_1}
		\vertiii{ (\bm E (\u) - \u_h,E(p)-p_h)}	\le \underbrace{\vertiii{ (\bm E (\u) - r_h\u, E(p)-\Pi_h p)}}_{\RomanNumeralCaps 1}	 + \underbrace{\vertiii{ ( r_h\u - \u_h,\Pi_h p-p_h)}}_{\RomanNumeralCaps 2}	.
	\end{align}	
	Theorem~\ref{cutfem:thm_approx} implies, for $s=\min\{t-1,r-1,k\}$,
	\begin{align}\label{eq:covid_2}
		\RomanNumeralCaps 1 \le C h^{s} \left(\norm{\u}_{H^{k}(\Omega)} + \norm{p}_{H^{r}(\Omega)}  \right).
	\end{align}	
	By Theorem~\ref{thm:cutfem_stability}, for $(\bm\u_h-r_h\bm\u,p_h-\pi_h p)$ there exists $\left(\vv_h,q_h\right)\in V_h\times Q_h$ such that
	\begin{align}\label{eq:covid_3}
		\RomanNumeralCaps 2=	\vertiii{(\bm\u_h-r_h\bm\u,p_h-\Pi_h p)}\lesssim 
		\frac{\mathcal A_h\left( \left(\bm\u_h-r_h\bm\u,p_h-\Pi_h p\right);\left(\vv_h,q_h\right)\right)}{\vertiii{(\vv_h,q_h)}}.
	\end{align}
	From Propositions~\ref{prop:weak_go},~\ref{prop:continuity}, and the definition of $\vertiii{\cdot}$, it holds
	\begin{equation}
		\begin{aligned}\label{eq:covid}
			\mathcal A_h\left( \left(\bm\u_h-r_h\bm\u,p_h-\Pi_h p\right);\left(\vv_h,q_h\right)\right) =& \mathcal A_h \left((\u_h - \u, p_h -p);(\vv_h,q_h) \right)\\
			&+  \mathcal A_h \left((\u - r_h\u, p -\Pi_h p);(\vv_h,q_h) \right)	\\
			\lesssim& \abs{\bm{j}_h (r_h\u,\vv_h)}+ \abs{j_h(\Pi_h p,q_h)}\\
			&+ \vertiii{(\u-r_h \u,p-\Pi_h p)}\vertiii{(\vv_h,q_h)} .
		\end{aligned}		
	\end{equation}	
	Lemmas~\ref{semi_cs},~\ref{lemma:bound_ghost} entail, for $s:= \min\{t-1,r-1,k\}$,
	\begin{equation}
		\begin{aligned}\label{eq:torneri}
			\abs{\bm{j}_h (r_h\u,\vv_h)}+ \abs{j_h(\Pi_h p,q_h)} \le &	\bm{j}_h (r_h\u,r_h\u)^{\frac{1}{2}} \bm{j}_h (\vv_h,\vv_h)^{\frac{1}{2}} + 	j_h (\Pi_h p,\Pi_h p)^{\frac{1}{2}} j_h (q_h,q_h)^{\frac{1}{2}} \\
			\lesssim & h^s \norm{\u}_{H^t(\Omega)} \norm{\vv_h}_{0,h,\Omega_\T} + h^s \norm{p}_{H^r(\Omega)}\norm{q_h}_{1,h,\Omega_\T}\\
			\lesssim & h^s \left(\norm{\u}_{H^t(\Omega)} + \norm{p}_{H^r(\Omega)} \right) \vertiii{(\vv_h,q_h)},
		\end{aligned}
	\end{equation}	
	where $s:= \min\{t-1,r-1,k\}$.
	By plugging~\eqref{eq:torneri} back into~\eqref{eq:covid} and using Theorem~\ref{cutfem:thm_approx}, we obtain
	\begin{equation}
		\begin{aligned}\label{eq:covid_4}
			\mathcal A_h\left( \left(\bm\u_h-r_h\bm\u,p_h-\Pi_h p\right);\left(\vv_h,q_h\right)\right) \lesssim & 		h^s \left(\norm{\u}_{H^t(\Omega)} + \norm{p}_{H^r(\Omega)} \right) \vertiii{(\vv_h,q_h)}.
		\end{aligned}
	\end{equation}			
	We combine~\eqref{eq:covid_1},~\eqref{eq:covid_2},~\eqref{eq:covid_3} and~\eqref{eq:covid_4}, getting
	\begin{align*}
		\vertiii{ (\bm E (\u) - \u_h,E(p)-p_h)} \lesssim &h^s \left(\norm{\u}_{H^{t}(\Omega)} +\norm{p}_{H^{r}(\Omega)}\right) ,
	\end{align*}
	where $s:= \min\{t-1,r-1,k\}$.
\end{proof}

\begin{remark}
	The convergence rates given by Theorem~\ref{cutfem:thm_approx} are optimal for the chosen discrete norms. On the other hand, the scaling of the energy norm $\norm{\cdot}_{0,h,\Omega_\T}$ does not allow us to obtain optimal convergence rates for the velocities with respect to the $L^2$-norm by simply applying Theorem~\ref{cutfem:thm_approx}. This is due to the term $h^{-1} \int_{\Gamma_N}\u_h\cdot\n\vv_h\cdot\n$: the natural weight, mimicking the $H^{-\frac{1}{2}}$-scalar product, would be $h$ instead of $h^{-1}$. However, as already discussed in~\cite{puppi_burman} such weight does not lead to an optimally converging scheme.
\end{remark}

\section{The condition number}\label{section5}
Following~\cite{MR2050138} and~\cite{MR3268662}, we want to show that the Euclidean condition number of the matrix arising from the discretization~\eqref{cutfem:disc_pb2} is uniformly bounded by $Ch^{-2}$, where $C>0$ is independent of how the underlying mesh cuts the boundary. Let us observe that the usual scaling of the condition number for a finite element discretization of the Darcy problem is $\mathcal O(h^{-1})$. We pay with a factor $h^{-1}$ because of the choice of the discrete norms.

Denoting $N=\dim V_h\times Q_h$, we can expand an arbitrary element $(\vv_h,q_h)\in V_h\times Q_h$ as $(\vv_h,q_h) = \sum_{i=1}^N \mathcal V_i \varphi_i$, where $\left(\varphi \right)_{i=1}^N$ is the finite element basis for the product space $V_h\times Q_h$ and $\mathcal V\in\R^N$ is its coordinate vector. The previous expansion of the elements of $V_h\times Q_h$ uniquely defines a canonical isometric isomorphism between  $V_h\times Q_h$ and $\R^N$, namely 
\begin{align*}
	\mathcal C: V_h \times Q_h \to \R^N, \qquad (\vv_h, q_h) \mapsto \mathcal V	.
\end{align*}	
Here,  $\R^N$ is equipped with the standard Euclidean scalar product, denoted as $\left( \cdot,\cdot \right)_{\ell^2}$, and the induced norm $\abs{\cdot}_{\ell^2}$. Given $M \in \R^{N\times N}$, we denote as $\norm{M}_2$ the matrix norm of $M$ induced by $\abs{\cdot}_{\ell^2}$. Let $A\in\R^{N\times N}$ denote the matrix associated to the discrete formulation~\eqref{cutfem:disc_pb2}, namely
\begin{align*}
	\left( A\mathcal V,\mathcal W\right)_{\ell^2}=\mathcal A_h \left( (\vv_h,q_h);(\w_h,r_h) \right), \qquad\forall\ \mathcal V, \mathcal W \in \R^N,
\end{align*}	
where $\mathcal V= \mathcal C (\vv_h,q_h)$, $\mathcal W = \mathcal C (\w_h,r_h)$ and $(\vv_h,q_h),(\w_h,r_h)\in V_h\times Q_h$. 
\begin{remark}
	In the pure Neumann case $\Gamma=\Gamma_N$, then the solution to~\eqref{cutfem:disc_pb2} for the pressure is determined up to a constant, hence $A$ is singular and $\ker A=\operatorname{span}\{ \mathcal C (\bm 0,1)\}$. We shall consider, instead of $A$, its bijective restriction $\restr{A}{\hat{\R}^N}: \hat {\R}^N \to \tilde {\R}^N$, where $\hat{\R}^N := \R^N/ \ker A$ and $\tilde{\R}^N:=\Im (A)$. 
\end{remark}
As already said, the goal is to analyze the Euclidean condition number $\kappa_2(A):=\norm{A}_2 \norm{A^{-1}}_2$. From~\cite{MR1299729}, we know that the main ingredients for the conditioning analysis are the following:
\begin{enumerate}
	\item the stability of the discrete formulation	with respect to a given norm $\norm{\cdot}_a$;
	\item the $\ell^2$-stability of the basis with respect to a given norm $\norm{\cdot}_b$;
	\item the equivalence between the norms $\norm{\cdot}_a$ and $\norm{\cdot}_b$.
\end{enumerate}	
In the subsequent exposition, the product norm~\eqref{triple:norm} plays the role of $\norm{\cdot}_b$, while $\vertiii{\cdot}_{L^2(\Omega_\T)}$, defined as $\vertiii{(\vv_h,q_h)}^2_{L^2(\Omega_{\T})} := \norm{\vv_h}^2_{L^2(\Omega_\T)} + \norm{q_h}^2_{L^2(\Omega_\T)}$, corresponds to $\norm{\cdot}_a$.
\begin{lemma} \label{cond:lemma1}
	There exist $C_1,C_2>0$ such that, for every $\left( \vv_h,q_h\right)\in V_h\times Q_h$,
	\begin{align*}
		C_1 h^{\frac{d}{2}} \abs{\mathcal V}_{\ell^2}	\le \vertiii{(\vv_h,q_h)}_{L^2(\Omega_{\T})} \le C_2 h^{\frac{d}{2}} \abs{\mathcal V}_{\ell^2},
	\end{align*}	
	where $\mathcal V=\mathcal C (\vv_h,q_h)$.
\end{lemma}
\begin{proof}
	The result holds because the background mesh is shape-regular and quasi-uniform. We refer the interested reader to Lemma~A.1 in~\cite{MR2223503}.	
\end{proof}	

\begin{lemma}\label{cond:lemma2}
	There exist $C_1, C_2>0$ such that, for every $\left( \vv_h,q_h\right)\in V_h\times Q_h$,
	\begin{align*}
		C_1 	h^{-1}\vertiii{(\vv_h,q_h)}_{L^2(\Omega_{\T})} \le \vertiii{(\vv_h,q_h)}\le  C_2	h^{-1}\vertiii{(\vv_h,q_h)}_{L^2(\Omega_{\T})}.
	\end{align*}		
\end{lemma}	
\begin{proof}
	The first bound follows because of a Poincar\'e-Friedrichs inequality for piecewise $H^1$-functions (Section~10.6 of~\cite{MR2373954}).	For the second inequality it is sufficient to apply standard inverse estimates for boundary-fitted finite elements.	
\end{proof}	

\begin{theorem}
	There exists $C>0$ such that
	\begin{align*}
		\kappa_2 (A)\le C h^{-2}.	
	\end{align*}	
\end{theorem}	
\begin{proof}
	Let us start by bounding $\norm{A}_{\ell^2}$. Given $(\vv_h,q_h), (\w_h,r_h)\in V_h\times Q_h$ such that $\mathcal C(\vv_h,q_h)=\mathcal V$, $\mathcal C(\w_h,r_h)=\mathcal W$, we have
	\begin{align*}
		(A\mathcal V, \mathcal W)_{\ell^2} =& \mathcal A_h \left (\vv_h,q_h); (\w_h,r_h) \right) \lesssim  \vertiii{(\vv_h,q_h)}  \vertiii{(\w_h,r_h)}\\
		&  \lesssim h^{-2}  \vertiii{(\vv_h,q_h)}_{L^2(\Omega_\T)}  \vertiii{(\w_h,r_h)}_{L^2(\Omega_\T)} 
		\lesssim h^{d-2} \abs{\mathcal V}_{\ell^2} \abs{\mathcal W}_{\ell^2}.
	\end{align*}
	In the previous inequalities we used, respectively, the continuity of $\mathcal A_h(\cdot;\cdot)$, Lemma~\ref{cond:lemma2} and Lemma~\ref{cond:lemma1}.
	Hence, $\norm{A}_2 \lesssim h^{d-2}$.
	We need to bound $\norm{A^{-1}}_2$. Since (the restriction of) $A^{-1}$ is invertible, we can write
	\begin{align}\label{cond:eq1}
		\norm{A^{-1}}_{2} = \sup_{\substack{\mathcal Y \in{\R}^N \setminus\{0\}}} \frac{\abs{A^{-1}\mathcal Y}_{\ell^2}}{\abs{\mathcal Y}_{\ell^2}} = \sup_{\substack{A\mathcal V \in {\R}^N \\ \mathcal V \in {\R} \setminus \{0\}  }} \frac{\abs{\mathcal V}}{\abs{A\mathcal V}} = \sup_{\substack{ A\mathcal V \in{\R}^N \\ \mathcal V\in {\R}^N\setminus \{0\}} } \sup_{\substack{  \mathcal W\in {\R}^N\setminus \{0\}}} \frac{\abs{\mathcal V}_{\ell^2}\abs{\mathcal W}_{\ell^2}}{\left(A\mathcal V,\mathcal W\right)_{\ell^2}}.
	\end{align}
	From Theorem~\ref{thm:cutfem_stability}, we know that, for every $\mathcal V\in {\R}^N\setminus\{\bm 0\}$, there exists $\mathcal W$ such that
	\begin{equation}
		\begin{aligned}\label{cond:eq2}
			\left( A\mathcal V,\mathcal W \right)_{\ell^2}	=&\mathcal{A}_h \left( (\vv_h,q_h); (\w_h,q_h) \right) \gtrsim \vertiii{(\vv_h,q_h)} \vertiii{(\w_h,r_h)}  \gtrsim \vertiii{(\vv_h,q_h)}_{L^2(\Omega_\T)} \vertiii{(\w_h,r_h)}_{L^2(\Omega_\T)}\\
			\gtrsim  & h^d \abs{\mathcal V}_{\ell^2} \abs{\mathcal W}_{\ell^2}.
		\end{aligned}	
	\end{equation}
	Moreover in the last two inequalities we used, respectively, Lemma~\ref{cond:lemma2} and Lemma~\ref{cond:lemma1}. By combining~\eqref{cond:eq1} and~\eqref{cond:eq2} we get $\norm{A^{-1}}_2 \lesssim h^{-d}$. Hence, we are done.
\end{proof}

\section{The purely Dirichlet case}\label{section6}
The goal of this section is to sketch the main steps required to analyze formulation~\eqref{cutfem:disc_pb} in the purely Dirichlet case. If $\Gamma=\Gamma_D$, then we consider the following Raviart-Thomas finite element discretization of problem~\eqref{prob:cont}.

Find $\left(\u_h,p_h\right)\in V_h\times Q_h$ such that
\begin{align}\label{dirichlet:pb2}
	\begin{aligned}
		\int_\Omega \u_h\cdot \vv_h + b_0(\vv_h,p_h) = \int_\Omega\f \cdot \vv_h + \int_\Gamma p_D \vv_h\cdot\n,\qquad&\forall\ \vv_h\in V_h, \\
		b_0(\u_h,q_h)  = \int_\Omega g q_h,\qquad&\forall\ q_h\in Q_h.
	\end{aligned}
\end{align}	
It is natural to equip the discrete spaces $V_h$ and $Q_h$ with $\norm{\cdot}_{H(\dive;\Omega)}$ and $\norm{\cdot}_{L^2(\Omega)}$, respectively. It is readily seen that formulation~\eqref{dirichlet:pb2} satisfies the standard stability estimates for saddle point problems (see the hypotheses of Theorem~5.2.5 of~\cite{MR3097958}). Hence \emph{a priori} error estimates can be obtained by standard techniques (again, we refer the reader to Section~5.2 of~\cite{MR3097958}). This time, the convergence rates are optimal because of the choice of the norms. However, the conditioning of the arising linear system will still strongly depend on the way the boundary cuts the mesh. As for the general case with mixed boundary conditions, we propose to cure this issue with a ghost penalty-based stabilization. 
\begin{align*}
	\tilde{\bm{j}}_h (\w_h,\vv_h):=& \sum_{f\in \mathcal F_h^{\Gamma}} \sum_{j=0}^k h^{2j+1} \int_f [\partial_n^j \w_h] [\partial_n^j \vv_h],\qquad \w_h,\vv_h \in V_h,\\
	\tilde{{j}}_h (r_h,q_h):=&  \sum_{f\in \mathcal F_h^{\Gamma}} \sum_{j=0}^k h^{2j+1} \int_f [\partial_n^j  r_h] [\partial_n^j  q_h],\qquad r_h,q_h \in Q_h.	
\end{align*}	
Let us observe that the ghost penalty operators scale differently than in~\eqref{eq:ghost_penalty} because of the different choices of the norms to the mixed case. The stabilized formulation reads as follows.

Find $\left(\u_h,p_h\right)\in V_h\times Q_h$ such that
\begin{align}\label{dirichlet:pb2_stab}
	\begin{aligned}
		\int_\Omega \u_h\cdot \vv_h +\tilde{\bm{j}}_h(\u_h,\vv_h) + b_0(\vv_h,p_h)  = \int_\Omega\f \cdot \vv_h + \int_\Gamma p_D \vv_h\cdot\n,\qquad&\forall\ \vv_h\in V_h, \\
		b_0(\u_h,q_h) +	\tilde{{j}}_h (p_h,q_h)  = \int_\Omega g q_h,\qquad&\forall\ q_h\in Q_h,
	\end{aligned}
\end{align}	
By mimicking the same lines of Section~\ref{section5}, it is possible to show that the condition number of~\eqref{dirichlet:pb2_stab} goes as $\mathcal O\left(h^{-1} \right)$.

%% file: Sections/numerical_experiments.tex
\section{Numerical examples}\label{section7}
In the following, we use the ghost penalty projection-based operators $\bm{s}_h(\cdot,\cdot)$ and $s_h(\cdot,\cdot)$ defined in Remark~\ref{rmk:ghost_penalty}. As already observed, this choice is convenient from the implementation point of view since it spares us to calculate the jumps of possibly high-order normal derivatives through the facets in the vicinity with the cut boundary. We limit the scope of our numerical investigations to the case of Cartesian quadrilateral meshes in $2$D. To integrate in the cut elements, we employ the strategy depicted in~\cite{MR3982623}: the cut elements are reparametrized using polynomials with the same approximation order of the Raviart-Thomas space employed for the space discretization.

\begin{figure}[!ht]
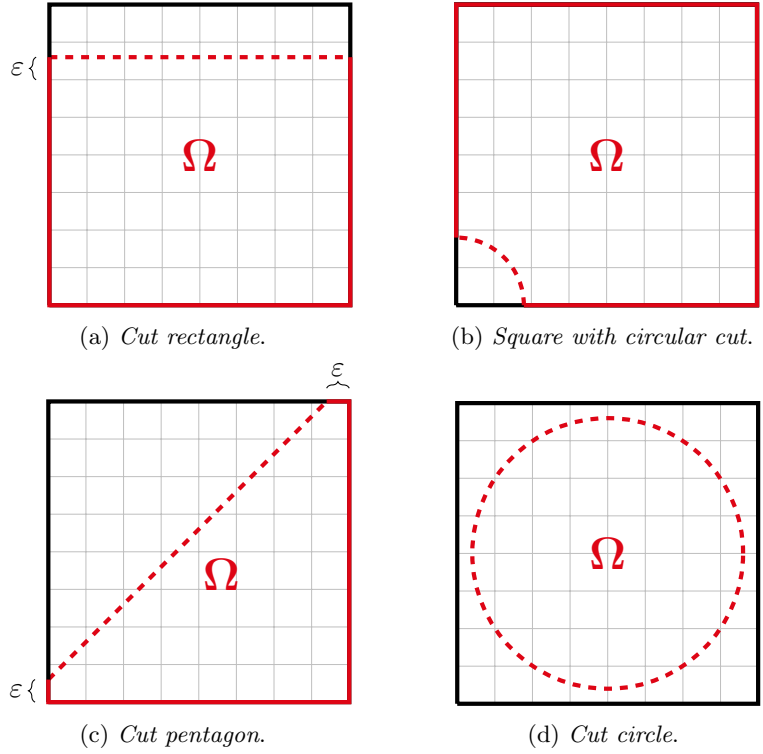

	\centering
	\subfloat[][\emph{Cut rectangle}.\label{fig:trimmed_square}]
	{
		\includestandalone[]{mesh_epsilon}
	}
	\hspace{1cm}
	\subfloat[][\emph{Square with circular cut}.\label{figure:fig5}]
	{
		\includestandalone[]{plate_with_hole}
	}
	\\
	\subfloat[][\emph{Cut pentagon}.\label{fig:trimmed_pentagon}]
	{
		\includestandalone[]{mesh_epsilon_pentagon}
	}
	\hspace{1cm}
	\subfloat[][\emph{Cut circle}.\label{figure:fig9}]
	{
		\includestandalone[]{circle}
	}
	\caption{Unfitted domains employed for the numerical experiments.}\label{figure:domains}
\end{figure}

\subsection{Convergence rates}

\subsubsection{Cut pentagon}
Let $\Omega_0=\left(0,1\right)^2$, $\Omega_1$ be the triangle with vertices $(0,0.25+\eps)-(0,1)-(0.75-\eps,1)$ and $\Omega=\Omega_0\setminus \overline\Omega_1$, with $\eps=10^{-9}$, see Figure~\ref{fig:trimmed_pentagon}. The reference solutions are
\begin{align*}
	\u_{ex}=
	\begin{pmatrix}
		y\sin(x)\cos(y)\\
		-x\sin(y)\cos(x)
	\end{pmatrix},
	\qquad
	p_{ex}= x^3y.
\end{align*}
Neumann boundary conditions are imposed on the whole boundary, weakly just on the sides that do not fit the underlying mesh. We compute the approximation errors of the velocity and pressure fields for different degrees $k\in\{0,1,2\}$, see Figure~\ref{figure:fig4}. We have optimal convergence, despite the sub-optimal result of Theorem~\ref{thm:apriori}. 

\begin{figure}[!ht]
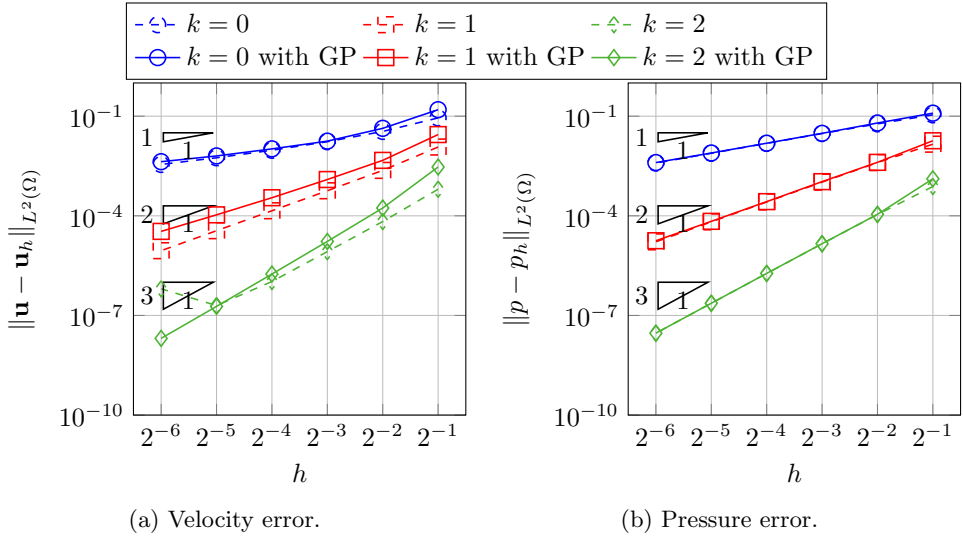

	\centering
	\includestandalone[]{legend_degrees2}  
	\\	
	\subfloat[][Velocity error.]
	{
		\includestandalone[]{error_velocity_pentagon_eps9_gp}
	}
	\subfloat[][Pressure error.]
	{
		\includestandalone[]{error_pressure_pentagon_eps9_gp}
	} 
	\caption{Convergence rates of the errors in the \emph{cut pentagon}.}\label{figure:fig4}
\end{figure}

\subsubsection{Cut circle}
Let us consider $\Omega= B_r(x_0)$, with $x_0=\left(0.5,0.5\right)$ and $r=0.45$, see Figure~\ref{figure:fig9}. The manufactured solution for the pressure is
\begin{align*}
	p_{ex}=\sin \left( 2\pi x \right) \cos \left( 2\pi \right), 
\end{align*}
and the velocity field is computed from Darcy's law~\eqref{prob:cont} when $\f$ is taken to be zero. We weakly prescribe Neumann boundary conditions on the whole boundary, which does not fit the underlying mesh. The $L^2$-errors for the velocity and pressure fields are plotted in Figure~\ref{figure:fig11}. We can see optimal orders of convergence and better accuracy in the stabilized case.

\begin{figure}[!ht]
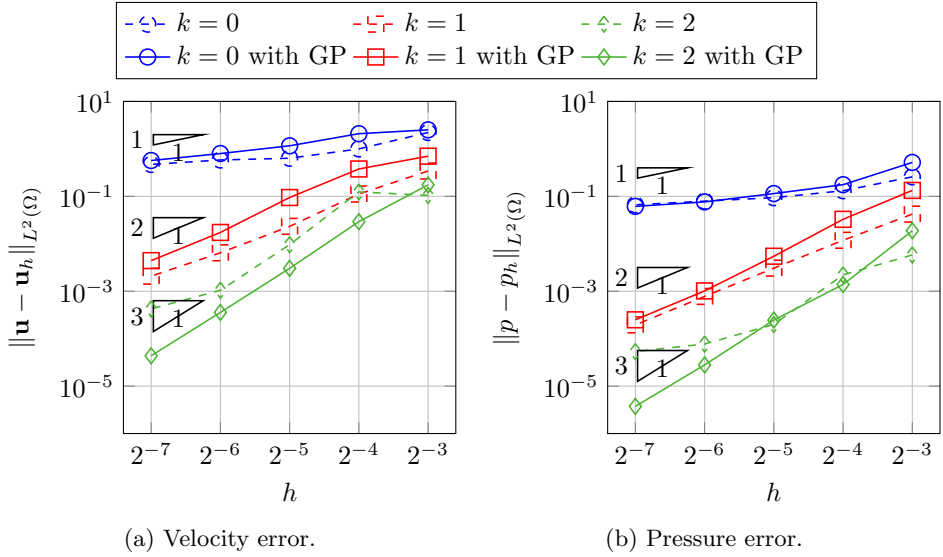

	\centering
	\includestandalone[]{legend_degrees2}  
	\\	
	\subfloat[][Velocity error.]
	{
		\includestandalone[]{error_velocity_circle}
	}
	\subfloat[][Pressure error.]
	{
		\includestandalone[]{error_pressure_circle}
	}
	\caption{Convergence rates of the errors in the \emph{cut circle}.}\label{figure:fig11} 
\end{figure}	

\subsection{Condition number}
\subsubsection{Cut rectangle}
\begin{figure}[!ht]
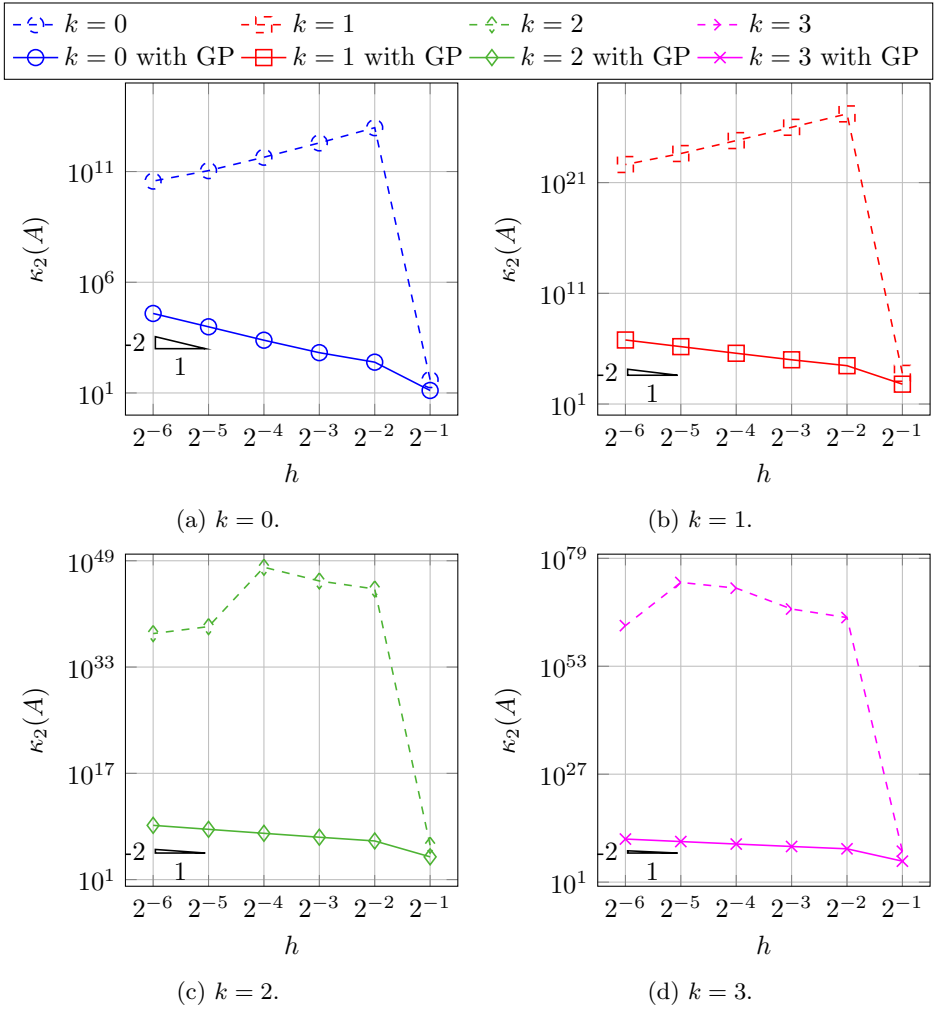

	\centering
	\includestandalone[]{legend_condnumb}  
	\\	
	\subfloat[][$k=0$.]
	{
		\includestandalone[]{condnumb_trimmedsquare_k=0}
	}
	\subfloat[][$k=1$.]
	{
		\includestandalone[]{condnumb_trimmedsquare_k=1}
	}\\
	\subfloat[][$k=2$.]
	{
		\includestandalone[]{condnumb_trimmedsquare_k=2}
	}
	\subfloat[][$k=3$.]
	{
		\includestandalone[]{condnumb_trimmedsquare_k=3}
	}
	\caption{Condition number for the \emph{cut rectangle} with Neumann boundary conditions.}\label{figure:fig80}
\end{figure}
\begin{figure}[!ht]
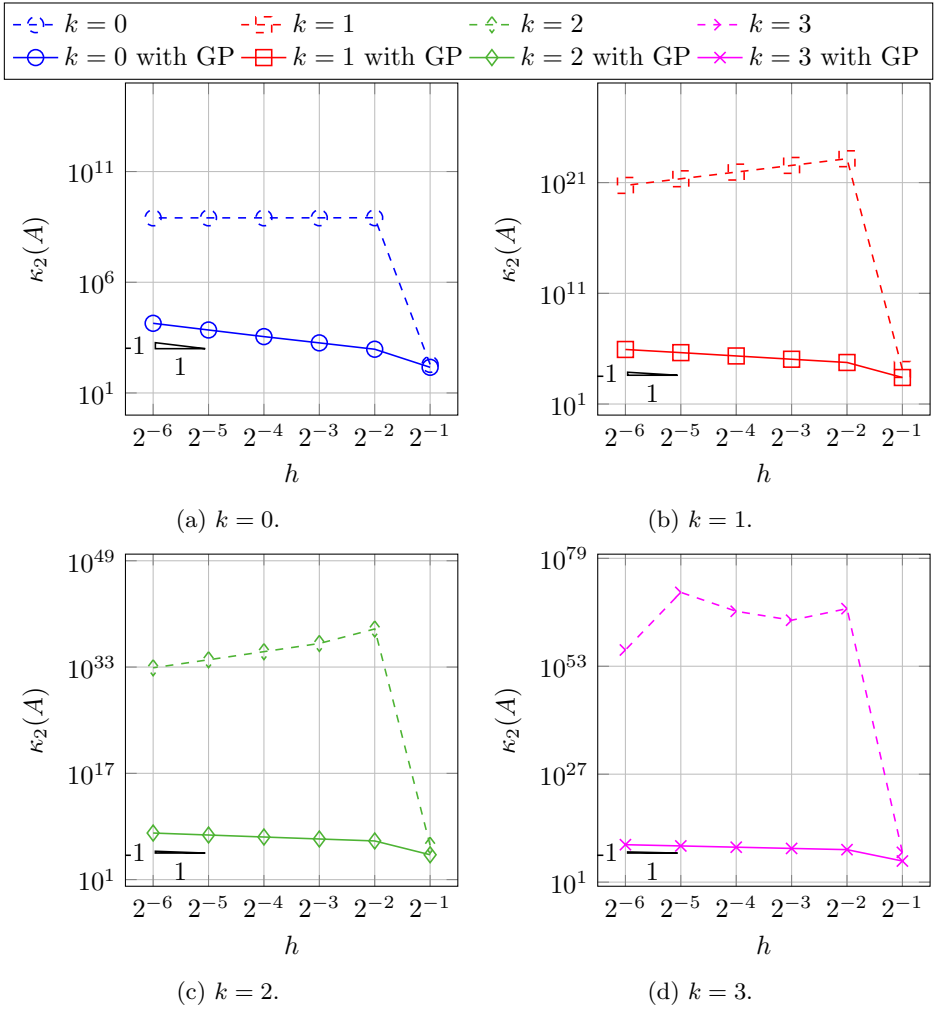

	\centering
	\includestandalone[]{legend_condnumb}  
	\\	
	\subfloat[][$k=0$.]
	{
		\includestandalone[]{condnumb_trimmedsquare_k=0_dirichlet}
	}
	\subfloat[][$k=1$.]
	{
		\includestandalone[]{condnumb_trimmedsquare_k=1_dirichlet}
	}\\
	\subfloat[][$k=2$.]
	{
		\includestandalone[]{condnumb_trimmedsquare_k=2_dirichlet}
	}
	\subfloat[][$k=3$.]
	{
		\includestandalone[]{condnumb_trimmedsquare_k=3_dirichlet}
	}
	\caption{Condition number for the \emph{cut rectangle} with Dirichlet boundary conditions.}\label{figure:fig80_dirichlet}
\end{figure}
Let us consider as physical domain the \emph{cut rectangle} $\Omega = \left(0,1\right) \times \left( 0,0.75+\eps\right)$ where $\eps=10^{-7}$, see Figure~\ref{fig:trimmed_square}. We impose Neumann boundary conditions weakly on the whole boundary. In Figure~\ref{figure:fig80} we compare the conditioning of the stabilized and non-stabilized formulations. Similarly, in Figure~\ref{figure:fig80_dirichlet} we compare the conditioning of the stabilized and non-stabilized formulations when Dirichlet boundary conditions are imposed. The results are in agreement with the theory developed in Sections~\ref{section5},~\ref{section6}. In particular, we observe that without stabilization, the condition number is negatively affected by the presence of cut elements and seems to grow without control, while in the stabilized case, the expected scaling of the conditioning is restored: $\mathcal O(h^{-2})$ for the Neumann case and $\mathcal O (h^{-1})$ for the purely Dirichlet case.

\subsection{On mass conservation}\label{numerical_experiment:mass}
Mass conservation is an important feature for finite element discretizations of incompressible flows, whose violation is not tolerable in many applications~\cite{MR3561143}. As observed in Remark~\ref{remark:mass_conservation}, the Raviart-Thomas finite element satisfies $\dive V_h =Q_h$ in the unfitted configuration as well. The formulation~\eqref{cutfem:disc_pb} as it stands is bound to fail to satisfy the incompressibility constraint in a weak sense, which is why, to exploit this property when the right-hand side $g$ vanishes, we consider the following non-symmetric variant of formulation~\eqref{cutfem:disc_pb}.

Find $\left( \u_h,p_h\right) \in V_h\times Q_h$ such that
\begin{equation}\label{cutfem:disc_pb_nonsym}
	\begin{aligned}	
		a_h(\u_h,\vv_h) +  b_1 (\vv_h,p_h) +\bm{j}_h(\u_h,\vv_h)  = \int_\Omega \f\cdot \vv_h + \int_{\Gamma_D} p_D \vv_h\cdot\n + h^{-1} \int_{\Gamma_N} u_N  \vv_h\cdot\n, \quad &\forall\ \vv_h\in V_h, \\
		b_0 (\u_h,q_h) +j_h(p_h,q_h)  =0,\quad & \forall\ q_h\in Q_h,
	\end{aligned}
\end{equation}
where
\begin{align*}
	b_0(\w_h,q_h):= \int_\Omega q_h \dive \w_h,\qquad \w_h\in V_h, q_h\in Q_h.
\end{align*}

Let us test formulation~\eqref{cutfem:disc_pb_nonsym} in the stabilized and non-stabilized cases. We take as reference solutions
\begin{align*}
	\u_{ex}=
	\begin{pmatrix}
		\cos(x)\operatorname{sinh}(y)\\
		\sin(x)\operatorname{cosh}(y)
	\end{pmatrix},
	\qquad
	p_{ex}=-\sin(x)\operatorname{sinh}(y)-\left(\cos(1)-1\right)\left(\operatorname{cosh}(1)-1\right).
\end{align*}
Note that $\dive \u_{ex} = 0$. We impose Dirichlet boundary conditions on $\{(x,y):x=2, 0\le y\le 2\}$ and on $\{(x,y) : 0\le x\le 2, y=2\}$, and weak Neumann boundary conditions on the rest of the boundary. The computed divergence of the discrete solution for the velocity is shown in Figure~\ref{figure:div_plate}. We observe that the ghost penalty stabilization pollutes the divergence of the velocity, hence also the non-symmetric formulation~\eqref{cutfem:disc_pb_nonsym} fails to the mass conservation at the discrete level.


\begin{figure}[!ht]
	\centering
	\subfloat[Without stabilizations.]
	{
		\includegraphics[width=0.33\textwidth,keepaspectratio=true]{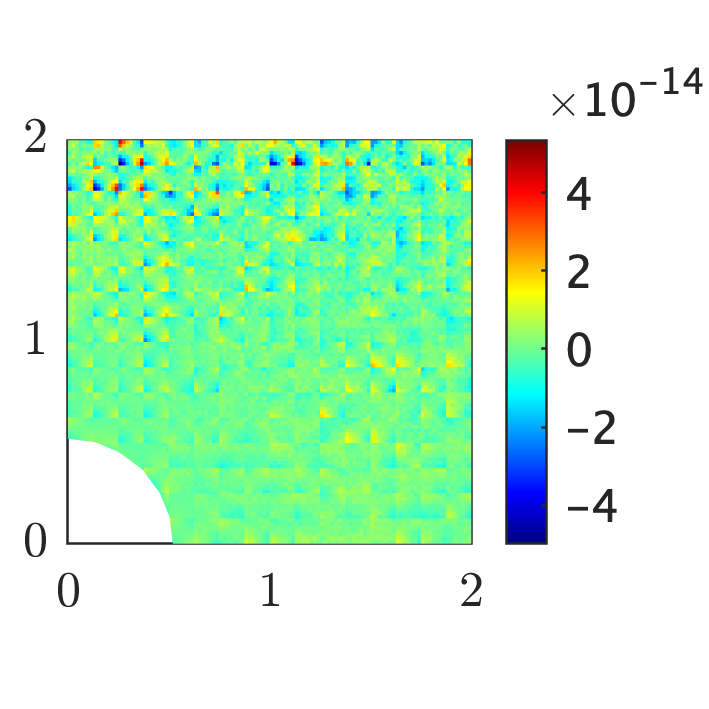}  \label{figure:div_plate1}
	}
	\subfloat[With ghost penalty.]
	{
		\includegraphics[width=0.33\textwidth,keepaspectratio=true]{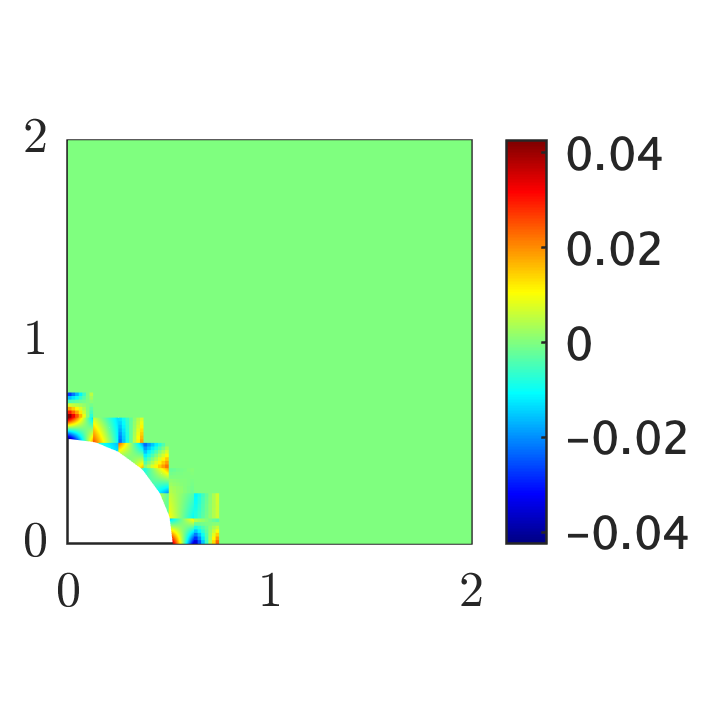} \label{figure:div_plate2}
	} 
	\caption{\emph{$\dive\u_h$ in the square with circular cut} obtained with~\eqref{cutfem:disc_pb_nonsym}.}\label{figure:div_plate}
\end{figure}

%% file: Figures/mesh_epsilon.tex
\begin{tikzpicture}
	
\definecolor{color2}{rgb}{0.872549019607843,0.019607843137255,0.0819607843137255}
\draw[step=0.5cm,lightgray,very thin] (0,0) grid (4,4);
\draw[dashed] (0,0) -- (1,0) ;
\draw[dashed] (0,0) -- (0,1) ;
\draw[color2, dashed, ultra thick] (0,3+0.3) -- (4,3+0.3);
\draw[ultra thick] (0,0) -- (4,0) -- (4,4) -- (0,4) -- (0,0);
\draw[color2,  ultra thick]  (0,3+0.3)--(0,0)--(4,0)--(4,3+0.3);
\coordinate (a) at (2,2);
\node at (a)   {\huge$\mathbf{\textcolor{color2}{\Omega}}$};
\begin{scope}
  \clip (0,0) -- (4,0) -- (4,4) -- (0,4) -- (0,0);
\end{scope}

   \draw[decoration={brace,raise=5pt},decorate]
  (0,3) -- node[left=6pt] {\large{$\varepsilon$}} (0,3+0.3);
\end{tikzpicture}

%% file: Figures/plate_with_hole.tex
\begin{tikzpicture}
\definecolor{color2}{rgb}{0.872549019607843,0.019607843137255,0.0819607843137255}

\draw[step=0.5cm,lightgray,very thin] (0,0) grid (4,4);
\draw[ultra thick] (0,0) -- (4,0) -- (4,4) -- (0,4) -- (0,0);
\draw[ultra thick] (0.9,0) -- (4,0) -- (4,4) -- (0,4) -- (0,0.9);
\draw[ultra thick, color=color2] (0.9,0) -- (4,0) -- (4,4) -- (0,4) -- (0,0.9);
\coordinate (a) at (2,2);
\node at (a)   {\huge$\mathbf{\textcolor{color2}{\Omega}}$};
\begin{scope}
  \clip (0,0) -- (4,0) -- (4,4) -- (0,4) -- (0,0);
  \draw[ultra thick, dashed, color=color2]  (0,0) circle (0.9cm);
\end{scope}

\end{tikzpicture}

%% file: Figures/mesh_epsilon_pentagon.tex
\begin{tikzpicture}
	
\definecolor{color2}{rgb}{0.872549019607843,0.019607843137255,0.0819607843137255}
\draw[step=0.5cm,lightgray,very thin] (0,0) grid (4,4);
\draw[dashed] (0,0) -- (1,0) ;
\draw[dashed] (0,0) -- (0,1) ;
\draw[color2, dashed, ultra thick] (0,0.3) -- (4-0.3,4);
\draw[ultra thick] (0,0) -- (4,0) -- (4,4) -- (0,4) -- (0,0);
\draw[color2,  ultra thick]  (0,0.3)--(0,0)--(4,0)--(4,4)--(4-0.3,4);
\coordinate (a) at (2+0.3,2-0.3);
\node at (a)   {\huge$\mathbf{\textcolor{color2}{\Omega}}$};
\begin{scope}
  \clip (0,0) -- (4,0) -- (4,4) -- (0,4) -- (0,0);
\end{scope}

\draw[decoration={brace,raise=5pt},decorate]
  (0,0) -- node[left=6pt] {\large{$\varepsilon$}} (0,0.3);
  
  \draw[decoration={brace,raise=5pt},decorate]
  (4-0.3,4) -- node[midway,yshift=0.42cm] {\large{$\varepsilon$}} (4,4);
\end{tikzpicture}

%% file: Figures/circle.tex
\begin{tikzpicture}

\definecolor{color2}{rgb}{0.872549019607843,0.019607843137255,0.0819607843137255}

\draw[step=0.5cm,lightgray,very thin] (0,0) grid (4,4);
\draw[ultra thick, color=black] (0,0) -- (4,0) -- (4,4) -- (0,4) -- (0,0);
\coordinate (a) at (2,2);
\node at (a)   {\huge$\mathbf{\textcolor{color2}{\Omega}}$};
\begin{scope}
	\clip (0,0) -- (4,0) -- (4,4) -- (0,4) -- (0,0);
	\draw[ultra thick, dashed, color=color2]  (a) circle (1.8cm);
\end{scope}

\end{tikzpicture}

%% file: Figures/legend_degrees2.tex
\newcommand{\logLogSlopeTriangleOne}[5]
{

    \pgfplotsextra
    {
        \pgfkeysgetvalue{/pgfplots/xmin}{\xmin}
        \pgfkeysgetvalue{/pgfplots/xmax}{\xmax}
        \pgfkeysgetvalue{/pgfplots/ymin}{\ymin}
        \pgfkeysgetvalue{/pgfplots/ymax}{\ymax}

        \pgfmathsetmacro{\xArel}{#1}
        \pgfmathsetmacro{\yArel}{#3}
        \pgfmathsetmacro{\xBrel}{#1-#2}
        \pgfmathsetmacro{\yBrel}{\yArel}
        \pgfmathsetmacro{\xCrel}{\xArel}

        \pgfmathsetmacro{\lnxB}{\xmin*(1-(#1-#2))+\xmax*(#1-#2)} 
        \pgfmathsetmacro{\lnxA}{\xmin*(1-#1)+\xmax*#1} 
        \pgfmathsetmacro{\lnyA}{\ymin*(1-#3)+\ymax*#3} 
        \pgfmathsetmacro{\lnyC}{\lnyA+#4*(\lnxA-\lnxB)}
        \pgfmathsetmacro{\yCrel}{\lnyC-\ymin)/(\ymax-\ymin)} 

        \coordinate (A) at (rel axis cs:\xArel,\yArel);
        \coordinate (B) at (rel axis cs:\xBrel,\yBrel);
        \coordinate (C) at (rel axis cs:\xCrel,\yCrel);

        \draw[#5]   (A)-- node[pos=0.5,anchor=north] {1} 
                    (B)-- 
                    (C)-- node[pos=0.5,anchor=east] {#4} 
                    cycle;
    }
}


\begin{tikzpicture}

\definecolor{blu}{rgb}{0,0,1}
\definecolor{red}{rgb}{1,0,0}
\definecolor{green}{rgb}{0.3,0.7,0.2}
\definecolor{mygray}{rgb}{0.5,0.5,0.5}
\definecolor{myyellow}{rgb}{0.9,0.7,0.1}
\definecolor{mymagenta}{rgb}{255,0,255}

  \begin{axis}[
  	legend columns=3,
  	width=5.5cm,
  	height=5.5cm,
      hide axis, 
      xmin=10,
      xmax=50,
      ymin=0,
      ymax=0.4,
            legend cell align={left},
      ]
\addlegendentry{$k=0$};
\addlegendimage{semithick, dashed, blu, mark=o, mark size=3, mark options={draw=blu}} 
\addlegendentry{$k=1$};
\addlegendimage{semithick, dashed, red, mark=square, mark size=3, mark options={draw=red}} 
\addlegendentry{$k=2$};
\addlegendimage{semithick, dashed, green, mark=diamond, mark size=3, mark options={draw=green}} 
\addlegendentry{$k=0$ with GP};
\addlegendimage{semithick, blu, mark=o, mark size=3, mark options={draw=blu}} 
\addlegendentry{$k=1$ with GP};
\addlegendimage{semithick, red, mark=square, mark size=3, mark options={draw=red}} 
\addlegendentry{$k=2$ with GP};
\addlegendimage{semithick, green, mark=diamond, mark size=3, mark options={draw=green}} 

\end{axis}

\end{tikzpicture}

%% file: Figures/error_velocity_pentagon_eps9_gp.tex
\newcommand{\logLogSlopeTriangleOne}[5]
{

    \pgfplotsextra
    {
        \pgfkeysgetvalue{/pgfplots/xmin}{\xmin}
        \pgfkeysgetvalue{/pgfplots/xmax}{\xmax}
        \pgfkeysgetvalue{/pgfplots/ymin}{\ymin}
        \pgfkeysgetvalue{/pgfplots/ymax}{\ymax}

        \pgfmathsetmacro{\xArel}{#1}
        \pgfmathsetmacro{\yArel}{#3}
        \pgfmathsetmacro{\xBrel}{#1-#2}
        \pgfmathsetmacro{\yBrel}{\yArel}
        \pgfmathsetmacro{\xCrel}{\xArel}

        \pgfmathsetmacro{\lnxB}{\xmin*(1-(#1-#2))+\xmax*(#1-#2)} 
        \pgfmathsetmacro{\lnxA}{\xmin*(1-#1)+\xmax*#1} 
        \pgfmathsetmacro{\lnyA}{\ymin*(1-#3)+\ymax*#3} 
        \pgfmathsetmacro{\lnyC}{\lnyA+#4*(\lnxA-\lnxB)}
        \pgfmathsetmacro{\yCrel}{\lnyC-\ymin)/(\ymax-\ymin)} 

        \coordinate (A) at (rel axis cs:\xArel,\yArel);
        \coordinate (B) at (rel axis cs:\xBrel,\yBrel);
        \coordinate (C) at (rel axis cs:\xCrel,\yCrel);

        \draw[#5]   (A)-- node[pos=0.5,anchor=north] {1} 
                    (B)-- 
                    (C)-- node[pos=0.5,anchor=east] {#4} 
                    cycle;
    }
}


\begin{tikzpicture}

\definecolor{blu}{rgb}{0,0,1}
\definecolor{red}{rgb}{1,0,0}
\definecolor{green}{rgb}{0.3,0.7,0.2}
\definecolor{mygray}{rgb}{0.5,0.5,0.5}
\definecolor{myyellow}{rgb}{0.9,0.7,0.1}
\definecolor{mymagenta}{rgb}{255,0,255}

\begin{loglogaxis}[
	width=6cm,
	height=6cm,
	xlabel={$h$},
	ylabel={$\norm{\u-\u_h}_{L^2(\Omega)}$},
	xtick=data,
	xticklabels={$2^{-1}$,$2^{-2}$, $2^{-3}$, $2^{-4}$, $2^{-5}$, $2^{-6}$},
	grid=major,
	ymin = 10^-10,
	ymax = 10^0,
	]

\addplot[semithick, blue, mark=o, mark size=3, mark options={draw=blue}] table[ x index = 0, y index  =1] {\dataDir/error_velocity_pentagon_eps9_gp.txt};
\addplot[semithick, dashed, blue, mark=o, mark size=3, mark options={draw=blue}] table[ x index = 0, y index  =1] {\dataDir/error_velocity_pentagon_eps9_nogp.txt};
\addplot[semithick, red, mark=square, mark size=3, mark options={draw=red}] table[ x index = 0, y index = 2] {\dataDir/error_velocity_pentagon_eps9_gp.txt};
\addplot[semithick, dashed, red, mark=square, mark size=3, mark options={draw=red}] table[ x index = 0, y index = 2] {\dataDir/error_velocity_pentagon_eps9_nogp.txt};
\addplot[semithick, green, mark=diamond, mark size=3, mark options={draw=green}] table[ x index = 0, y index = 3] {\dataDir/error_velocity_pentagon_eps9_gp.txt};
\addplot[semithick, dashed, green, mark=diamond, mark size=3, mark options={draw=green}] table[ x index = 0, y index = 3] {\dataDir/error_velocity_pentagon_eps9_nogp.txt};
\logLogSlopeTriangleOne{0.09}{-0.15}{0.85}{1}{black, semithick} ;
\logLogSlopeTriangleOne{0.09}{-0.15}{0.63}{2}{black, semithick} ;
\logLogSlopeTriangleOne{0.09}{-0.15}{0.40}{3}{black, semithick} ;

\end{loglogaxis}

\end{tikzpicture}

%% file: Figures/error_pressure_pentagon_eps9_gp.tex
\newcommand{\logLogSlopeTriangleOne}[5]
{

    \pgfplotsextra
    {
        \pgfkeysgetvalue{/pgfplots/xmin}{\xmin}
        \pgfkeysgetvalue{/pgfplots/xmax}{\xmax}
        \pgfkeysgetvalue{/pgfplots/ymin}{\ymin}
        \pgfkeysgetvalue{/pgfplots/ymax}{\ymax}

        \pgfmathsetmacro{\xArel}{#1}
        \pgfmathsetmacro{\yArel}{#3}
        \pgfmathsetmacro{\xBrel}{#1-#2}
        \pgfmathsetmacro{\yBrel}{\yArel}
        \pgfmathsetmacro{\xCrel}{\xArel}

        \pgfmathsetmacro{\lnxB}{\xmin*(1-(#1-#2))+\xmax*(#1-#2)} 
        \pgfmathsetmacro{\lnxA}{\xmin*(1-#1)+\xmax*#1} 
        \pgfmathsetmacro{\lnyA}{\ymin*(1-#3)+\ymax*#3} 
        \pgfmathsetmacro{\lnyC}{\lnyA+#4*(\lnxA-\lnxB)}
        \pgfmathsetmacro{\yCrel}{\lnyC-\ymin)/(\ymax-\ymin)} 

        \coordinate (A) at (rel axis cs:\xArel,\yArel);
        \coordinate (B) at (rel axis cs:\xBrel,\yBrel);
        \coordinate (C) at (rel axis cs:\xCrel,\yCrel);

        \draw[#5]   (A)-- node[pos=0.5,anchor=north] {1} 
                    (B)-- 
                    (C)-- node[pos=0.5,anchor=east] {#4} 
                    cycle;
    }
}


\begin{tikzpicture}

\definecolor{blu}{rgb}{0,0,1}
\definecolor{red}{rgb}{1,0,0}
\definecolor{green}{rgb}{0.3,0.7,0.2}
\definecolor{mygray}{rgb}{0.5,0.5,0.5}
\definecolor{myyellow}{rgb}{0.9,0.7,0.1}
\definecolor{mymagenta}{rgb}{255,0,255}

\begin{loglogaxis}[
	width=6cm,
	height=6cm,
	xlabel={$h$},
	ylabel={$\norm{p-p_h}_{L^2(\Omega)}$},
	xtick=data,
	xticklabels={$2^{-1}$,$2^{-2}$, $2^{-3}$, $2^{-4}$, $2^{-5}$, $2^{-6}$},
	grid=major,
	ymin = 10^-10,
	ymax = 10^0,
	]
	
\addplot[semithick, blue, mark=o, mark size=3, mark options={draw=blue}] table[ x index = 0, y index  =1] {\dataDir/error_pressure_pentagon_eps9_gp.txt};
\addplot[semithick, dashed, blue, mark=o, mark size=3, mark options={draw=blue}] table[ x index = 0, y index  =1] {\dataDir/error_pressure_pentagon_eps9_nogp.txt};
\addplot[semithick, red, mark=square, mark size=3, mark options={draw=red}] table[ x index = 0, y index = 2] {\dataDir/error_pressure_pentagon_eps9_gp.txt};
\addplot[semithick, dashed, red, mark=square, mark size=3, mark options={draw=red}] table[ x index = 0, y index = 2] {\dataDir/error_pressure_pentagon_eps9_nogp.txt};
\addplot[semithick, green, mark=diamond, mark size=3, mark options={draw=green}] table[ x index = 0, y index = 3] {\dataDir/error_pressure_pentagon_eps9_gp.txt};
\addplot[semithick, dashed, green, mark=diamond, mark size=3, mark options={draw=green}] table[ x index = 0, y index = 3] {\dataDir/error_pressure_pentagon_eps9_nogp.txt};
	\logLogSlopeTriangleOne{0.09}{-0.15}{0.85}{1}{black, semithick} ;
	\logLogSlopeTriangleOne{0.09}{-0.15}{0.63}{2}{black, semithick} ;
	\logLogSlopeTriangleOne{0.09}{-0.15}{0.40}{3}{black, semithick} ;

\end{loglogaxis}

\end{tikzpicture}

%% file: Figures/error_velocity_circle.tex
\newcommand{\logLogSlopeTriangleOne}[5]
{

    \pgfplotsextra
    {
        \pgfkeysgetvalue{/pgfplots/xmin}{\xmin}
        \pgfkeysgetvalue{/pgfplots/xmax}{\xmax}
        \pgfkeysgetvalue{/pgfplots/ymin}{\ymin}
        \pgfkeysgetvalue{/pgfplots/ymax}{\ymax}

        \pgfmathsetmacro{\xArel}{#1}
        \pgfmathsetmacro{\yArel}{#3}
        \pgfmathsetmacro{\xBrel}{#1-#2}
        \pgfmathsetmacro{\yBrel}{\yArel}
        \pgfmathsetmacro{\xCrel}{\xArel}

        \pgfmathsetmacro{\lnxB}{\xmin*(1-(#1-#2))+\xmax*(#1-#2)} 
        \pgfmathsetmacro{\lnxA}{\xmin*(1-#1)+\xmax*#1} 
        \pgfmathsetmacro{\lnyA}{\ymin*(1-#3)+\ymax*#3} 
        \pgfmathsetmacro{\lnyC}{\lnyA+#4*(\lnxA-\lnxB)}
        \pgfmathsetmacro{\yCrel}{\lnyC-\ymin)/(\ymax-\ymin)} 

        \coordinate (A) at (rel axis cs:\xArel,\yArel);
        \coordinate (B) at (rel axis cs:\xBrel,\yBrel);
        \coordinate (C) at (rel axis cs:\xCrel,\yCrel);

        \draw[#5]   (A)-- node[pos=0.5,anchor=north] {1} 
                    (B)-- 
                    (C)-- node[pos=0.5,anchor=east] {#4} 
                    cycle;
    }
}


\begin{tikzpicture}

\definecolor{blu}{rgb}{0,0,1}
\definecolor{red}{rgb}{1,0,0}
\definecolor{green}{rgb}{0.3,0.7,0.2}
\definecolor{mygray}{rgb}{0.5,0.5,0.5}
\definecolor{myyellow}{rgb}{0.9,0.7,0.1}
\definecolor{mymagenta}{rgb}{255,0,255}

\begin{loglogaxis}[
	width=6cm,
	height=6cm,
	xlabel={$h$},
	ylabel={$\norm{\u-\u_h}_{L^2(\Omega)}$},
	xtick=data,
	xticklabels={$2^{-3}$,$2^{-4}$, $2^{-5}$, $2^{-6}$, $2^{-7}$},
	grid=major,
	ymin = 10^-(6),
	ymax = 10^1,
	]

\addplot[semithick, blue, mark=o, mark size=3, mark options={draw=blue}] table[ x index = 0, y index  =1] {\dataDir/error_velocity_circle_gp.txt};
\addplot[semithick, dashed, blue, mark=o, mark size=3, mark options={draw=blue}] table[ x index = 0, y index  =1] {\dataDir/error_velocity_circle_nogp.txt};
\addplot[semithick, red, mark=square, mark size=3, mark options={draw=red}] table[ x index = 0, y index = 2] {\dataDir/error_velocity_circle_gp.txt};
\addplot[semithick, red, dashed, mark=square, mark size=3, mark options={draw=red}] table[ x index = 0, y index = 2] {\dataDir/error_velocity_circle_nogp.txt};
\addplot[semithick, green, mark=diamond, mark size=3, mark options={draw=green}] table[ x index = 0, y index = 3] {\dataDir/error_velocity_circle_gp.txt};
\addplot[semithick, dashed, green, mark=diamond, mark size=3, mark options={draw=green}] table[ x index = 0, y index = 3] {\dataDir/error_velocity_circle_nogp.txt};
\logLogSlopeTriangleOne{0.09}{-0.15}{0.9}{1}{black, semithick} ;
\logLogSlopeTriangleOne{0.09}{-0.15}{0.65}{2}{black, semithick} ;
\logLogSlopeTriangleOne{0.09}{-0.15}{0.40}{3}{black, semithick} ;
\end{loglogaxis}

\end{tikzpicture}

%% file: Figures/error_pressure_circle.tex
\newcommand{\logLogSlopeTriangleOne}[5]
{

    \pgfplotsextra
    {
        \pgfkeysgetvalue{/pgfplots/xmin}{\xmin}
        \pgfkeysgetvalue{/pgfplots/xmax}{\xmax}
        \pgfkeysgetvalue{/pgfplots/ymin}{\ymin}
        \pgfkeysgetvalue{/pgfplots/ymax}{\ymax}

        \pgfmathsetmacro{\xArel}{#1}
        \pgfmathsetmacro{\yArel}{#3}
        \pgfmathsetmacro{\xBrel}{#1-#2}
        \pgfmathsetmacro{\yBrel}{\yArel}
        \pgfmathsetmacro{\xCrel}{\xArel}

        \pgfmathsetmacro{\lnxB}{\xmin*(1-(#1-#2))+\xmax*(#1-#2)} 
        \pgfmathsetmacro{\lnxA}{\xmin*(1-#1)+\xmax*#1} 
        \pgfmathsetmacro{\lnyA}{\ymin*(1-#3)+\ymax*#3} 
        \pgfmathsetmacro{\lnyC}{\lnyA+#4*(\lnxA-\lnxB)}
        \pgfmathsetmacro{\yCrel}{\lnyC-\ymin)/(\ymax-\ymin)} 

        \coordinate (A) at (rel axis cs:\xArel,\yArel);
        \coordinate (B) at (rel axis cs:\xBrel,\yBrel);
        \coordinate (C) at (rel axis cs:\xCrel,\yCrel);

        \draw[#5]   (A)-- node[pos=0.5,anchor=north] {1} 
                    (B)-- 
                    (C)-- node[pos=0.5,anchor=east] {#4} 
                    cycle;
    }
}


\begin{tikzpicture}

\definecolor{blu}{rgb}{0,0,1}
\definecolor{red}{rgb}{1,0,0}
\definecolor{green}{rgb}{0.3,0.7,0.2}
\definecolor{mygray}{rgb}{0.5,0.5,0.5}
\definecolor{myyellow}{rgb}{0.9,0.7,0.1}
\definecolor{mymagenta}{rgb}{255,0,255}

\begin{loglogaxis}[
	width=6cm,
	height=6cm,
	xlabel={$h$},
	ylabel={$\norm{p-p_h}_{L^2(\Omega)}$},
	xtick=data,
	xticklabels={$2^{-3}$, $2^{-4}$, $2^{-5}$, $2^{-6}$, $2^{-7}$},
	grid=major,
	ymin = 10^-6,
	ymax = 10^1,
	]

\addplot[semithick, blue, mark=o, mark size=3, mark options={draw=blue}] table[ x index = 0, y index  =1] {\dataDir/error_pressure_circle_gp.txt};
\addplot[semithick, dashed, blue, mark=o, mark size=3, mark options={draw=blue}] table[ x index = 0, y index  =1] {\dataDir/error_pressure_circle_nogp.txt};
\addplot[semithick, red, mark=square, mark size=3, mark options={draw=red}] table[ x index = 0, y index = 2] {\dataDir/error_pressure_circle_gp.txt};
\addplot[semithick, dashed, red, mark=square, mark size=3, mark options={draw=red}] table[ x index = 0, y index = 2] {\dataDir/error_pressure_circle_nogp.txt};
\addplot[semithick, green, mark=diamond, mark size=3, mark options={draw=green}] table[ x index = 0, y index = 3] {\dataDir/error_pressure_circle_gp.txt};
\addplot[semithick, dashed, green, mark=diamond, mark size=3, mark options={draw=green}] table[ x index = 0, y index = 3] {\dataDir/error_pressure_circle_nogp.txt};
\logLogSlopeTriangleOne{0.09}{-0.15}{0.80}{1}{black, semithick} ;
\logLogSlopeTriangleOne{0.09}{-0.15}{0.50}{2}{black, semithick} ;
\logLogSlopeTriangleOne{0.09}{-0.15}{0.25}{3}{black, semithick} ;
\end{loglogaxis}

\end{tikzpicture}

%% file: Figures/legend_condnumb.tex
\newcommand{\logLogSlopeTriangleOne}[5]
{

    \pgfplotsextra
    {
        \pgfkeysgetvalue{/pgfplots/xmin}{\xmin}
        \pgfkeysgetvalue{/pgfplots/xmax}{\xmax}
        \pgfkeysgetvalue{/pgfplots/ymin}{\ymin}
        \pgfkeysgetvalue{/pgfplots/ymax}{\ymax}

        \pgfmathsetmacro{\xArel}{#1}
        \pgfmathsetmacro{\yArel}{#3}
        \pgfmathsetmacro{\xBrel}{#1-#2}
        \pgfmathsetmacro{\yBrel}{\yArel}
        \pgfmathsetmacro{\xCrel}{\xArel}

        \pgfmathsetmacro{\lnxB}{\xmin*(1-(#1-#2))+\xmax*(#1-#2)} 
        \pgfmathsetmacro{\lnxA}{\xmin*(1-#1)+\xmax*#1} 
        \pgfmathsetmacro{\lnyA}{\ymin*(1-#3)+\ymax*#3} 
        \pgfmathsetmacro{\lnyC}{\lnyA+#4*(\lnxA-\lnxB)}
        \pgfmathsetmacro{\yCrel}{\lnyC-\ymin)/(\ymax-\ymin)} 

        \coordinate (A) at (rel axis cs:\xArel,\yArel);
        \coordinate (B) at (rel axis cs:\xBrel,\yBrel);
        \coordinate (C) at (rel axis cs:\xCrel,\yCrel);

        \draw[#5]   (A)-- node[pos=0.5,anchor=north] {1} 
                    (B)-- 
                    (C)-- node[pos=0.5,anchor=east] {#4} 
                    cycle;
    }
}


\begin{tikzpicture}

\definecolor{blu}{rgb}{0,0,1}
\definecolor{red}{rgb}{1,0,0}
\definecolor{green}{rgb}{0.3,0.7,0.2}
\definecolor{mygray}{rgb}{0.5,0.5,0.5}
\definecolor{myyellow}{rgb}{0.9,0.7,0.1}
\definecolor{mymagenta}{rgb}{255,0,255}

  \begin{axis}[
  	legend columns=4,
  	width=5.5cm,
  	height=5.5cm,
      hide axis, 
      xmin=10,
      xmax=50,
      ymin=0,
      ymax=0.4,
            legend cell align={left},
      ]


\addlegendentry{$k=0$};
\addlegendimage{semithick, dashed, blue, mark=o, mark size=3, mark options={draw=blu}} 
\addlegendentry{$k=1$};
\addlegendimage{semithick, dashed, red, mark=square, mark size=3, mark options={draw=red}} 
\addlegendentry{$k=2$};
\addlegendimage{semithick, dashed, green, mark=diamond, mark size=3, mark options={draw=green}} 
\addlegendentry{$k=3$};
\addlegendimage{semithick, dashed, mymagenta, mark=x, mark size=3, mark options={draw=mymagenta}} 
\addlegendentry{$k=0$ with GP};
\addlegendimage{semithick, blu, mark=o, mark size=3, mark options={draw=blu}} 
\addlegendentry{$k=1$ with GP};
\addlegendimage{semithick, red, mark=square, mark size=3, mark options={draw=red}} 
\addlegendentry{$k=2$ with GP};
\addlegendimage{semithick, green, mark=diamond, mark size=3, mark options={draw=green}} 
\addlegendentry{$k=3$ with GP};
\addlegendimage{semithick, mymagenta, mark=x, mark size=3, mark options={draw=mymagenta}} 
\end{axis}

\end{tikzpicture}

%% file: Figures/condnumb_trimmedsquare_k=0.tex
\newcommand{\logLogSlopeTriangleOne}[5]
{

    \pgfplotsextra
    {
        \pgfkeysgetvalue{/pgfplots/xmin}{\xmin}
        \pgfkeysgetvalue{/pgfplots/xmax}{\xmax}
        \pgfkeysgetvalue{/pgfplots/ymin}{\ymin}
        \pgfkeysgetvalue{/pgfplots/ymax}{\ymax}

        \pgfmathsetmacro{\xArel}{#1}
        \pgfmathsetmacro{\yArel}{#3}
        \pgfmathsetmacro{\xBrel}{#1-#2}
        \pgfmathsetmacro{\yBrel}{\yArel}
        \pgfmathsetmacro{\xCrel}{\xArel}

        \pgfmathsetmacro{\lnxB}{\xmin*(1-(#1-#2))+\xmax*(#1-#2)} 
        \pgfmathsetmacro{\lnxA}{\xmin*(1-#1)+\xmax*#1} 
        \pgfmathsetmacro{\lnyA}{\ymin*(1-#3)+\ymax*#3} 
        \pgfmathsetmacro{\lnyC}{\lnyA+#4*(\lnxA-\lnxB)}
        \pgfmathsetmacro{\yCrel}{\lnyC-\ymin)/(\ymax-\ymin)} 

        \coordinate (A) at (rel axis cs:\xArel,\yArel);
        \coordinate (B) at (rel axis cs:\xBrel,\yBrel);
        \coordinate (C) at (rel axis cs:\xCrel,\yCrel);

        \draw[#5]   (A)-- node[pos=0.5,anchor=north] {1} 
                    (B)-- 
                    (C)-- node[pos=0.5,anchor=east] {#4} 
                    cycle;
    }
}


\begin{tikzpicture}

\definecolor{blu}{rgb}{0,0,1}
\definecolor{red}{rgb}{1,0,0}
\definecolor{green}{rgb}{0.3,0.7,0.2}
\definecolor{mygray}{rgb}{0.5,0.5,0.5}
\definecolor{myyellow}{rgb}{0.9,0.7,0.1}
\definecolor{mymagenta}{rgb}{255,0,255}

\begin{loglogaxis}[
	width=6cm,
	height=6cm,
	xlabel={$h$},
	ylabel={$\kappa_2(A)$},
	xtick=data,
	xticklabels={$2^{-1}$,$2^{-2}$, $2^{-3}$, $2^{-4}$, $2^{-5}$, $2^{-6}$},
	grid=major,
	ymin = 10^0,
	ymax = 10^15,
	]

\addplot[semithick, dashed, blue, mark=o, mark size=3, mark options={draw=blue}] table[ x index = 0, y index  =1] {\dataDir/condnumb_trimmedsquare.txt};
\addplot[semithick, blue, mark=o, mark size=3, mark options={draw=blue}] table[ x index = 0, y index  =2] {\dataDir/condnumb_trimmedsquare.txt};
\logLogSlopeTriangleOne{0.09}{-0.15}{0.20}{-2}{black, semithick} ;
\end{loglogaxis}

\end{tikzpicture}

%% file: Figures/condnumb_trimmedsquare_k=1.tex
\newcommand{\logLogSlopeTriangleOne}[5]
{

    \pgfplotsextra
    {
        \pgfkeysgetvalue{/pgfplots/xmin}{\xmin}
        \pgfkeysgetvalue{/pgfplots/xmax}{\xmax}
        \pgfkeysgetvalue{/pgfplots/ymin}{\ymin}
        \pgfkeysgetvalue{/pgfplots/ymax}{\ymax}

        \pgfmathsetmacro{\xArel}{#1}
        \pgfmathsetmacro{\yArel}{#3}
        \pgfmathsetmacro{\xBrel}{#1-#2}
        \pgfmathsetmacro{\yBrel}{\yArel}
        \pgfmathsetmacro{\xCrel}{\xArel}

        \pgfmathsetmacro{\lnxB}{\xmin*(1-(#1-#2))+\xmax*(#1-#2)} 
        \pgfmathsetmacro{\lnxA}{\xmin*(1-#1)+\xmax*#1} 
        \pgfmathsetmacro{\lnyA}{\ymin*(1-#3)+\ymax*#3} 
        \pgfmathsetmacro{\lnyC}{\lnyA+#4*(\lnxA-\lnxB)}
        \pgfmathsetmacro{\yCrel}{\lnyC-\ymin)/(\ymax-\ymin)} 

        \coordinate (A) at (rel axis cs:\xArel,\yArel);
        \coordinate (B) at (rel axis cs:\xBrel,\yBrel);
        \coordinate (C) at (rel axis cs:\xCrel,\yCrel);

        \draw[#5]   (A)-- node[pos=0.5,anchor=north] {1} 
                    (B)-- 
                    (C)-- node[pos=0.5,anchor=east] {#4} 
                    cycle;
    }
}


\begin{tikzpicture}

\definecolor{blu}{rgb}{0,0,1}
\definecolor{red}{rgb}{1,0,0}
\definecolor{green}{rgb}{0.3,0.7,0.2}
\definecolor{mygray}{rgb}{0.5,0.5,0.5}
\definecolor{myyellow}{rgb}{0.9,0.7,0.1}
\definecolor{mymagenta}{rgb}{255,0,255}

\begin{loglogaxis}[
	width=6cm,
	height=6cm,
	xlabel={$h$},
	ylabel={$\kappa_2(A)$},
	xtick=data,
	xticklabels={$2^{-1}$,$2^{-2}$, $2^{-3}$, $2^{-4}$, $2^{-5}$, $2^{-6}$},
	grid=major,
	ymin = 10^0,
	ymax = 10^30,
	]
\addplot[semithick, dashed, red, mark=square, mark size=3, mark options={draw=red}] table[ x index = 0, y index = 3] {\dataDir/condnumb_trimmedsquare.txt};
\addplot[semithick, red, mark=square, mark size=3, mark options={draw=red}] table[ x index = 0, y index = 4] {\dataDir/condnumb_trimmedsquare.txt};
\logLogSlopeTriangleOne{0.09}{-0.15}{0.12}{-2}{black, semithick} ;
\end{loglogaxis}

\end{tikzpicture}

%% file: Figures/condnumb_trimmedsquare_k=2.tex
\newcommand{\logLogSlopeTriangleOne}[5]
{

    \pgfplotsextra
    {
        \pgfkeysgetvalue{/pgfplots/xmin}{\xmin}
        \pgfkeysgetvalue{/pgfplots/xmax}{\xmax}
        \pgfkeysgetvalue{/pgfplots/ymin}{\ymin}
        \pgfkeysgetvalue{/pgfplots/ymax}{\ymax}

        \pgfmathsetmacro{\xArel}{#1}
        \pgfmathsetmacro{\yArel}{#3}
        \pgfmathsetmacro{\xBrel}{#1-#2}
        \pgfmathsetmacro{\yBrel}{\yArel}
        \pgfmathsetmacro{\xCrel}{\xArel}

        \pgfmathsetmacro{\lnxB}{\xmin*(1-(#1-#2))+\xmax*(#1-#2)} 
        \pgfmathsetmacro{\lnxA}{\xmin*(1-#1)+\xmax*#1} 
        \pgfmathsetmacro{\lnyA}{\ymin*(1-#3)+\ymax*#3} 
        \pgfmathsetmacro{\lnyC}{\lnyA+#4*(\lnxA-\lnxB)}
        \pgfmathsetmacro{\yCrel}{\lnyC-\ymin)/(\ymax-\ymin)} 

        \coordinate (A) at (rel axis cs:\xArel,\yArel);
        \coordinate (B) at (rel axis cs:\xBrel,\yBrel);
        \coordinate (C) at (rel axis cs:\xCrel,\yCrel);

        \draw[#5]   (A)-- node[pos=0.5,anchor=north] {1} 
                    (B)-- 
                    (C)-- node[pos=0.5,anchor=east] {#4} 
                    cycle;
    }
}


\begin{tikzpicture}

\definecolor{blu}{rgb}{0,0,1}
\definecolor{red}{rgb}{1,0,0}
\definecolor{green}{rgb}{0.3,0.7,0.2}
\definecolor{mygray}{rgb}{0.5,0.5,0.5}
\definecolor{myyellow}{rgb}{0.9,0.7,0.1}
\definecolor{mymagenta}{rgb}{255,0,255}

\begin{loglogaxis}[
	width=6cm,
	height=6cm,
	xlabel={$h$},
	ylabel={$\kappa_2(A)$},
	xtick=data,
	xticklabels={$2^{-1}$,$2^{-2}$, $2^{-3}$, $2^{-4}$, $2^{-5}$, $2^{-6}$},
	grid=major,
	ymin = 10^0,
	ymax = 10^50,
	]
\addplot[semithick, dashed, green, mark=diamond, mark size=3, mark options={draw=green}] table[ x index = 0, y index = 5] {\dataDir/condnumb_trimmedsquare.txt};
\addplot[semithick, green, mark=diamond, mark size=3, mark options={draw=green}] table[ x index = 0, y index = 6] {\dataDir/condnumb_trimmedsquare.txt};
\logLogSlopeTriangleOne{0.09}{-0.15}{0.10}{-2}{black, semithick} ;
\end{loglogaxis}

\end{tikzpicture}

%% file: Figures/condnumb_trimmedsquare_k=3.tex
\newcommand{\logLogSlopeTriangleOne}[5]
{

    \pgfplotsextra
    {
        \pgfkeysgetvalue{/pgfplots/xmin}{\xmin}
        \pgfkeysgetvalue{/pgfplots/xmax}{\xmax}
        \pgfkeysgetvalue{/pgfplots/ymin}{\ymin}
        \pgfkeysgetvalue{/pgfplots/ymax}{\ymax}

        \pgfmathsetmacro{\xArel}{#1}
        \pgfmathsetmacro{\yArel}{#3}
        \pgfmathsetmacro{\xBrel}{#1-#2}
        \pgfmathsetmacro{\yBrel}{\yArel}
        \pgfmathsetmacro{\xCrel}{\xArel}

        \pgfmathsetmacro{\lnxB}{\xmin*(1-(#1-#2))+\xmax*(#1-#2)} 
        \pgfmathsetmacro{\lnxA}{\xmin*(1-#1)+\xmax*#1} 
        \pgfmathsetmacro{\lnyA}{\ymin*(1-#3)+\ymax*#3} 
        \pgfmathsetmacro{\lnyC}{\lnyA+#4*(\lnxA-\lnxB)}
        \pgfmathsetmacro{\yCrel}{\lnyC-\ymin)/(\ymax-\ymin)} 

        \coordinate (A) at (rel axis cs:\xArel,\yArel);
        \coordinate (B) at (rel axis cs:\xBrel,\yBrel);
        \coordinate (C) at (rel axis cs:\xCrel,\yCrel);

        \draw[#5]   (A)-- node[pos=0.5,anchor=north] {1} 
                    (B)-- 
                    (C)-- node[pos=0.5,anchor=east] {#4} 
                    cycle;
    }
}


\begin{tikzpicture}

\definecolor{blu}{rgb}{0,0,1}
\definecolor{red}{rgb}{1,0,0}
\definecolor{green}{rgb}{0.3,0.7,0.2}
\definecolor{mygray}{rgb}{0.5,0.5,0.5}
\definecolor{myyellow}{rgb}{0.9,0.7,0.1}
\definecolor{mymagenta}{rgb}{255,0,255}

\begin{loglogaxis}[
	width=6cm,
	height=6cm,
	xlabel={$h$},
	ylabel={$\kappa_2(A)$},
	xtick=data,
	xticklabels={$2^{-1}$,$2^{-2}$, $2^{-3}$, $2^{-4}$, $2^{-5}$, $2^{-6}$},
	grid=major,
	ymin = 10^0,
	ymax = 10^80,
	]
\addplot[semithick, dashed, mymagenta, mark=x, mark size=3, mark options={draw=mymagenta}] table[ x index = 0, y index = 7] {\dataDir/condnumb_trimmedsquare.txt};
\addplot[semithick, mymagenta, mark=x, mark size=3, mark options={draw=mymagenta}] table[ x index = 0, y index = 8] {\dataDir/condnumb_trimmedsquare.txt};
\logLogSlopeTriangleOne{0.09}{-0.15}{0.10}{-2}{black, semithick} ;
\end{loglogaxis}

\end{tikzpicture}

%% file: Figures/condnumb_trimmedsquare_k=0_dirichlet.tex
\newcommand{\logLogSlopeTriangleOne}[5]
{

    \pgfplotsextra
    {
        \pgfkeysgetvalue{/pgfplots/xmin}{\xmin}
        \pgfkeysgetvalue{/pgfplots/xmax}{\xmax}
        \pgfkeysgetvalue{/pgfplots/ymin}{\ymin}
        \pgfkeysgetvalue{/pgfplots/ymax}{\ymax}

        \pgfmathsetmacro{\xArel}{#1}
        \pgfmathsetmacro{\yArel}{#3}
        \pgfmathsetmacro{\xBrel}{#1-#2}
        \pgfmathsetmacro{\yBrel}{\yArel}
        \pgfmathsetmacro{\xCrel}{\xArel}

        \pgfmathsetmacro{\lnxB}{\xmin*(1-(#1-#2))+\xmax*(#1-#2)} 
        \pgfmathsetmacro{\lnxA}{\xmin*(1-#1)+\xmax*#1} 
        \pgfmathsetmacro{\lnyA}{\ymin*(1-#3)+\ymax*#3} 
        \pgfmathsetmacro{\lnyC}{\lnyA+#4*(\lnxA-\lnxB)}
        \pgfmathsetmacro{\yCrel}{\lnyC-\ymin)/(\ymax-\ymin)} 

        \coordinate (A) at (rel axis cs:\xArel,\yArel);
        \coordinate (B) at (rel axis cs:\xBrel,\yBrel);
        \coordinate (C) at (rel axis cs:\xCrel,\yCrel);

        \draw[#5]   (A)-- node[pos=0.5,anchor=north] {1} 
                    (B)-- 
                    (C)-- node[pos=0.5,anchor=east] {#4} 
                    cycle;
    }
}


\begin{tikzpicture}

\definecolor{blu}{rgb}{0,0,1}
\definecolor{red}{rgb}{1,0,0}
\definecolor{green}{rgb}{0.3,0.7,0.2}
\definecolor{mygray}{rgb}{0.5,0.5,0.5}
\definecolor{myyellow}{rgb}{0.9,0.7,0.1}
\definecolor{mymagenta}{rgb}{255,0,255}

\begin{loglogaxis}[
	width=6cm,
	height=6cm,
	xlabel={$h$},
	ylabel={$\kappa_2(A)$},
	xtick=data,
	xticklabels={$2^{-1}$,$2^{-2}$, $2^{-3}$, $2^{-4}$, $2^{-5}$, $2^{-6}$},
	grid=major,
	ymin = 10^0,
	ymax = 10^15,
	]

\addplot[semithick, dashed, blue, mark=o, mark size=3, mark options={draw=blue}] table[ x index = 0, y index  =1] {\dataDir/condnumb_trimmedsquare_dirichlet.txt};
\addplot[semithick, blue, mark=o, mark size=3, mark options={draw=blue}] table[ x index = 0, y index  =2] {\dataDir/condnumb_trimmedsquare_dirichlet.txt};
\logLogSlopeTriangleOne{0.09}{-0.15}{0.20}{-1}{black, semithick} ;
\end{loglogaxis}

\end{tikzpicture}

%% file: Figures/condnumb_trimmedsquare_k=1_dirichlet.tex
\newcommand{\logLogSlopeTriangleOne}[5]
{

    \pgfplotsextra
    {
        \pgfkeysgetvalue{/pgfplots/xmin}{\xmin}
        \pgfkeysgetvalue{/pgfplots/xmax}{\xmax}
        \pgfkeysgetvalue{/pgfplots/ymin}{\ymin}
        \pgfkeysgetvalue{/pgfplots/ymax}{\ymax}

        \pgfmathsetmacro{\xArel}{#1}
        \pgfmathsetmacro{\yArel}{#3}
        \pgfmathsetmacro{\xBrel}{#1-#2}
        \pgfmathsetmacro{\yBrel}{\yArel}
        \pgfmathsetmacro{\xCrel}{\xArel}

        \pgfmathsetmacro{\lnxB}{\xmin*(1-(#1-#2))+\xmax*(#1-#2)} 
        \pgfmathsetmacro{\lnxA}{\xmin*(1-#1)+\xmax*#1} 
        \pgfmathsetmacro{\lnyA}{\ymin*(1-#3)+\ymax*#3} 
        \pgfmathsetmacro{\lnyC}{\lnyA+#4*(\lnxA-\lnxB)}
        \pgfmathsetmacro{\yCrel}{\lnyC-\ymin)/(\ymax-\ymin)} 

        \coordinate (A) at (rel axis cs:\xArel,\yArel);
        \coordinate (B) at (rel axis cs:\xBrel,\yBrel);
        \coordinate (C) at (rel axis cs:\xCrel,\yCrel);

        \draw[#5]   (A)-- node[pos=0.5,anchor=north] {1} 
                    (B)-- 
                    (C)-- node[pos=0.5,anchor=east] {#4} 
                    cycle;
    }
}


\begin{tikzpicture}

\definecolor{blu}{rgb}{0,0,1}
\definecolor{red}{rgb}{1,0,0}
\definecolor{green}{rgb}{0.3,0.7,0.2}
\definecolor{mygray}{rgb}{0.5,0.5,0.5}
\definecolor{myyellow}{rgb}{0.9,0.7,0.1}
\definecolor{mymagenta}{rgb}{255,0,255}

\begin{loglogaxis}[
	width=6cm,
	height=6cm,
	xlabel={$h$},
	ylabel={$\kappa_2(A)$},
	xtick=data,
	xticklabels={$2^{-1}$,$2^{-2}$, $2^{-3}$, $2^{-4}$, $2^{-5}$, $2^{-6}$},
	grid=major,
	ymin = 10^0,
	ymax = 10^30,
	]
\addplot[semithick, dashed, red, mark=square, mark size=3, mark options={draw=red}] table[ x index = 0, y index = 3] {\dataDir/condnumb_trimmedsquare_dirichlet.txt};
\addplot[semithick, red, mark=square, mark size=3, mark options={draw=red}] table[ x index = 0, y index = 4] {\dataDir/condnumb_trimmedsquare_dirichlet.txt};
\logLogSlopeTriangleOne{0.09}{-0.15}{0.12}{-1}{black, semithick} ;
\end{loglogaxis}

\end{tikzpicture}

%% file: Figures/condnumb_trimmedsquare_k=2_dirichlet.tex
\newcommand{\logLogSlopeTriangleOne}[5]
{

    \pgfplotsextra
    {
        \pgfkeysgetvalue{/pgfplots/xmin}{\xmin}
        \pgfkeysgetvalue{/pgfplots/xmax}{\xmax}
        \pgfkeysgetvalue{/pgfplots/ymin}{\ymin}
        \pgfkeysgetvalue{/pgfplots/ymax}{\ymax}

        \pgfmathsetmacro{\xArel}{#1}
        \pgfmathsetmacro{\yArel}{#3}
        \pgfmathsetmacro{\xBrel}{#1-#2}
        \pgfmathsetmacro{\yBrel}{\yArel}
        \pgfmathsetmacro{\xCrel}{\xArel}

        \pgfmathsetmacro{\lnxB}{\xmin*(1-(#1-#2))+\xmax*(#1-#2)} 
        \pgfmathsetmacro{\lnxA}{\xmin*(1-#1)+\xmax*#1} 
        \pgfmathsetmacro{\lnyA}{\ymin*(1-#3)+\ymax*#3} 
        \pgfmathsetmacro{\lnyC}{\lnyA+#4*(\lnxA-\lnxB)}
        \pgfmathsetmacro{\yCrel}{\lnyC-\ymin)/(\ymax-\ymin)} 

        \coordinate (A) at (rel axis cs:\xArel,\yArel);
        \coordinate (B) at (rel axis cs:\xBrel,\yBrel);
        \coordinate (C) at (rel axis cs:\xCrel,\yCrel);

        \draw[#5]   (A)-- node[pos=0.5,anchor=north] {1} 
                    (B)-- 
                    (C)-- node[pos=0.5,anchor=east] {#4} 
                    cycle;
    }
}


\begin{tikzpicture}

\definecolor{blu}{rgb}{0,0,1}
\definecolor{red}{rgb}{1,0,0}
\definecolor{green}{rgb}{0.3,0.7,0.2}
\definecolor{mygray}{rgb}{0.5,0.5,0.5}
\definecolor{myyellow}{rgb}{0.9,0.7,0.1}
\definecolor{mymagenta}{rgb}{255,0,255}

\begin{loglogaxis}[
	width=6cm,
	height=6cm,
	xlabel={$h$},
	ylabel={$\kappa_2(A)$},
	xtick=data,
	xticklabels={$2^{-1}$,$2^{-2}$, $2^{-3}$, $2^{-4}$, $2^{-5}$, $2^{-6}$},
	grid=major,
	ymin = 10^0,
	ymax = 10^50,
	]
\addplot[semithick, dashed, green, mark=diamond, mark size=3, mark options={draw=green}] table[ x index = 0, y index = 5] {\dataDir/condnumb_trimmedsquare_dirichlet.txt};
\addplot[semithick, green, mark=diamond, mark size=3, mark options={draw=green}] table[ x index = 0, y index = 6] {\dataDir/condnumb_trimmedsquare_dirichlet.txt};
\logLogSlopeTriangleOne{0.09}{-0.15}{0.10}{-1}{black, semithick} ;
\end{loglogaxis}

\end{tikzpicture}

%% file: Figures/condnumb_trimmedsquare_k=3_dirichlet.tex
\newcommand{\logLogSlopeTriangleOne}[5]
{

    \pgfplotsextra
    {
        \pgfkeysgetvalue{/pgfplots/xmin}{\xmin}
        \pgfkeysgetvalue{/pgfplots/xmax}{\xmax}
        \pgfkeysgetvalue{/pgfplots/ymin}{\ymin}
        \pgfkeysgetvalue{/pgfplots/ymax}{\ymax}

        \pgfmathsetmacro{\xArel}{#1}
        \pgfmathsetmacro{\yArel}{#3}
        \pgfmathsetmacro{\xBrel}{#1-#2}
        \pgfmathsetmacro{\yBrel}{\yArel}
        \pgfmathsetmacro{\xCrel}{\xArel}

        \pgfmathsetmacro{\lnxB}{\xmin*(1-(#1-#2))+\xmax*(#1-#2)} 
        \pgfmathsetmacro{\lnxA}{\xmin*(1-#1)+\xmax*#1} 
        \pgfmathsetmacro{\lnyA}{\ymin*(1-#3)+\ymax*#3} 
        \pgfmathsetmacro{\lnyC}{\lnyA+#4*(\lnxA-\lnxB)}
        \pgfmathsetmacro{\yCrel}{\lnyC-\ymin)/(\ymax-\ymin)} 

        \coordinate (A) at (rel axis cs:\xArel,\yArel);
        \coordinate (B) at (rel axis cs:\xBrel,\yBrel);
        \coordinate (C) at (rel axis cs:\xCrel,\yCrel);

        \draw[#5]   (A)-- node[pos=0.5,anchor=north] {1} 
                    (B)-- 
                    (C)-- node[pos=0.5,anchor=east] {#4} 
                    cycle;
    }
}


\begin{tikzpicture}

\definecolor{blu}{rgb}{0,0,1}
\definecolor{red}{rgb}{1,0,0}
\definecolor{green}{rgb}{0.3,0.7,0.2}
\definecolor{mygray}{rgb}{0.5,0.5,0.5}
\definecolor{myyellow}{rgb}{0.9,0.7,0.1}
\definecolor{mymagenta}{rgb}{255,0,255}

\begin{loglogaxis}[
	width=6cm,
	height=6cm,
	xlabel={$h$},
	ylabel={$\kappa_2(A)$},
	xtick=data,
	xticklabels={$2^{-1}$,$2^{-2}$, $2^{-3}$, $2^{-4}$, $2^{-5}$, $2^{-6}$},
	grid=major,
	ymin = 10^0,
	ymax = 10^80,
	]
\addplot[semithick, dashed, mymagenta, mark=x, mark size=3, mark options={draw=mymagenta}] table[ x index = 0, y index = 7] {\dataDir/condnumb_trimmedsquare_dirichlet.txt};
\addplot[semithick, mymagenta, mark=x, mark size=3, mark options={draw=mymagenta}] table[ x index = 0, y index = 8] {\dataDir/condnumb_trimmedsquare_dirichlet.txt};
\logLogSlopeTriangleOne{0.09}{-0.15}{0.10}{-1}{black, semithick} ;
\end{loglogaxis}

\end{tikzpicture}

%% file: Sections/appendix.tex
\section{Auxiliary results}

\begin{lemma}\label{lemma:trace_ineq}
	There exists $C>0$ depending on $\Gamma$, but not on the way it cuts the mesh, such that for every $K\in\mathcal G_h$:
	\begin{equation*}
		\norm{v}^2_{L^2(K\cap\Gamma)}\le C \norm{v}_{L^2(K)}\norm{v}_{H^1(K)},\qquad\forall\ v\in H^1(K).
	\end{equation*}	
\end{lemma}	
\begin{proof}
	See, for instance, Lemma~3 in~\cite{MR1941489}, Lemma~3 in~\cite{MR2075053}, or Lemma~4.1 of~\cite{MR3407236}.
\end{proof}	
\begin{lemma}\label{lemma:disc_trace_ineq}
There exists $C>0$ depending on $\Gamma$, but not on the way it cuts the mesh, such that for every $K\in\mathcal G_h$:
\begin{align*}
h^{\frac{1}{2}} \norm{\vv_h\cdot\n}_{L^2(\Gamma_K)} \le C  \norm{\vv_h}_{L^2(K)},\qquad\forall\ \vv_h\in V_h.
\end{align*}	
\end{lemma}	
\begin{proof}
It follows by Lemma~\ref{lemma:trace_ineq}, finite dimensionality and a scaling argument.
\end{proof}